\documentclass[12pt]{article}
\usepackage{amsthm,amsmath,amssymb}
\usepackage[notref,notcite]{showkeys}
\usepackage[numbers,sort&compress]{natbib}
\allowdisplaybreaks[4]

\usepackage[colorlinks,
linkcolor=blue,
anchorcolor=blue,
citecolor=blue
]{hyperref}

\newtheorem{thm}{Theorem}[section]

\newtheorem*{thm*}{Theorem*}
\newtheorem{cor}[thm]{Corollary}
\newtheorem{lem}[thm]{Lemma}
\newtheorem{prop}{Proposition}[section]

\numberwithin{equation}{section}
\theoremstyle{remark}
\newtheorem{rem}{Remark}

\theoremstyle{definition}
\newtheorem*{defn}{Definition}

\def\oz{\omega}
\def\lz{\lambda}

\def\Oz{\Omega}

\def\dz{\delta}
\def\az{\alpha}
\def\gz{\gamma}

\def\tz{\theta}

\def\({\Bigl(}
\def \){ \Bigr)}

\def\sub{\substack}

 \def\az{{\alpha}}
 
 \def\gz{{\gamma}}
 
 \def\tz{{\theta}}
 \def\lz{{\lambda}}
 \def\dz{{\delta}}
 \def\oz{{\omega}}

\def\ss{{\Bbb S}^{d}}

\begin{document}

\title{Optimal quadrature for weighted function spaces on multivariate domains
\thanks{The research was
supported by the National Natural Science Foundation of China
(Project no. 12371098).}
}
\author{Jiansong Li\thanks{E-mail address: cnuljs2023@163.com},\quad Heping Wang\thanks{Corresponding author E-mail address: wanghp@cnu.edu.cn}
\\
\small  School of Mathematical Sciences, Capital Normal
University, Beijing 100048\\
\small People's Republic of China
}
\date{}
\maketitle {\bf Abstract:} Consider the numerical integration
$${\rm Int}_{\mathbb S^d,w}(f)=\int_{\mathbb S^d}f({\bf x})w({\bf
x}){\rm d}\sigma({\bf x}) $$ for weighted Sobolev classes
$BW_{p,w}^r(\mathbb S^d)$ with  a Dunkl weight $w$ and weighted
Besov classes $BB_\gamma^\Theta(L_{p,w}(\mathbb S^d))$ with the
generalized smoothness index $\Theta $  and  a doubling weight $w$
on the unit sphere $\mathbb S^d$ of the Euclidean space $\mathbb
R^{d+1}$   in the deterministic and randomized case settings. For
$BW_{p,w}^r(\mathbb S^d)$  we obtain the optimal quadrature errors
in both settings. For $BB_\gamma^\Theta(L_{p,w}(\mathbb S^d))$ we
use the weighted least $\ell_p$ approximation and the standard
Monte Carlo algorithm to
 obtain  upper estimates of the quadrature errors  which are
optimal  if $w$ is an $A_\infty$ weight in the deterministic case
setting or  if $w$ is a product weight in the randomized case
setting. Our results show that randomized  algorithms can provide
a faster convergence rate  than that of the deterministic ones
when $p>1$.

Similar results are also established on the unit ball and the standard simplex of  $\mathbb R^d$.

{\bf Keywords:} {Numerical integration; Weighted Sobolev classes;
Weighted Besov classes of generalized smoothness; Deterministic
and randomized case settings; Weighted least $\ell_p$
approximation}

{\bf MSC(2000) subject classification:} 65D30; 65D32; 41A55;
65C05; 33C50; 33C52.

\maketitle

\section{Introduction}\label{sec1}

A critical problem in applied mathematics and data sciences is to calculate the numerical integration
\begin{equation}\label{1.1}{\rm Int}_{\Omega^d,w}(f)=\int_{\Omega^d}f({\bf x})w({\bf x}){\rm d}\mathsf m({\bf x}),\end{equation}
for a function $f$ belonging to a class $F_d$ of continuous
functions on $\Omega^d$, where $\Omega^d$ is a compact subset of
the Euclidean space equipped with a measure $w({\bf x}){\rm
d}\mathsf m({\bf x})$. For most cases,  we could not compute this
integration utilizing the fundamental theorem of the calculus
since there is no closed form expression of the antiderivatives.
Therefore, we have to approximate such integration numerically.
Meanwhile, in the applications we only know function values on
finite points  on $\Omega^d$ which are called \emph{standard
information}. Hence, we have to use algorithms based on these
finite function values.

One type of algorithms is the \emph{deterministic algorithm} of
the form
$$A_n(f):=\varphi_n\left[f({\bf x}_1), f({\bf x}_2),\dots,f({\bf x}_n)\right],$$where ${\bf x}_j\in \Oz^d$ can be chosen adaptively and $\varphi_n:\mathbb R^n\rightarrow\mathbb R$
is an arbitrary mapping. We denoted by $\mathcal{A}_n^{\rm det}$
the class of  all algorithms of this form. (Adaption means that the selection of
${\bf x}_j$ may depend on the already computed values $f({\bf x}_1),
f({\bf x}_2),\dots,f({\bf x}_{j-1})$.)
Then, for a positive integer $n$, we define
\begin{itemize}
  \item the \emph{deterministic case error} of an algorithm $A_n$ on $F_d$  by
\begin{equation*}
  e^{\rm det}({\rm Int}_{\Oz^d,w};F_d,\,A_n):=\sup\limits_{f\in F_d}\left|{\rm Int}_{\Oz^d,w}(f)-A_n(f)\right|;
\end{equation*}
  \item the \emph{$n$-th minimal (optimal) deterministic case error} on $F_d$  by
\begin{equation*}
  e_n^{\rm det}({\rm Int}_{\Oz^d,w};F_d):=\inf\limits_{A_n\in\mathcal{A}_n^{\rm det}}e^{\rm det}({\rm Int}_{\Oz^d,w};F_d,\,A_n).
\end{equation*}
\end{itemize}
 Usually we assume that the class $F_d$ is convex and
balanced. It follows from \cite{B3} that the $n$-th minimal
deterministic case error $e_n^{\rm det}({\rm Int}_{\Oz^d,w};F_d)$
can be achieved by linear algorithms, i.e., the algorithms of the
form
$$A_n^{\rm lin}(f)=\sum_{j=1}^n\lz_jf({\bf x}_j),\ \lz_j\in\mathbb R,\
j=1,\dots,n,$$which is also called a \emph{quadrature formula},
and, moreover, it follows from \cite[Theorem~4.7]{NW2008} that
\begin{equation}\label{1}
e_n^{\rm det}({\rm Int}_{\Oz^d,w};F_d)= \inf_{\sub{{\bf x}_1,\dots,{\bf x}_n\in \Omega^d}}\sup_{\sub{f\in F_d\\
  f({\bf x}_1)=\dots=f({\bf x}_n)=0}} \int_{\Omega^d}f({\bf x})w({\bf x}){\rm d}\mathsf m({\bf x}).
\end{equation}We say that a quadrature formula $A_n^{\rm lin}$ is \emph{positive} if $\lz_j\ge 0,j=1,\dots,n$, and is a
\emph{quasi-Monte-Carlo (QMC) algorithm} if $\lz_j=1/n,j=1,\dots,n$.

For usual  smooth classes the integration problem in the
deterministic setting  suffers from the curse of dimensionality if
the dimension $d$ is large. It is well known that   the standard
Monte-Carlo algorithm overcomes  the curse of dimensionality. In
this paper we consider  {randomized algorithms} (Monte-Carlo
algorithms). \emph{Randomized algorithms} (also called Monte-Carlo
algorithms) are understood as $\Sigma\otimes \mathcal B(F_d)$
measurable functions
$$(A^\oz)=(A^\oz(\cdot))_{\oz\in\mathcal{F}}: \mathcal{F}\times  F_d\to
\mathbb R,$$ where $\mathcal B( F_d)$  denotes the  Borel
$\sigma$-algebra of $ F_d$, $(\mathcal{F}, \Sigma, \mathcal P)$ is
a suitable probability space, and for any fixed $\oz\in
\mathcal{F}$, $A^\oz$ is a deterministic method with cardinality
$n(f,\oz)$. The number $n(f,\oz)$  may be randomized and
adaptively depend on the input, and the cardinality of $(A^\oz)$
is then defined by $${\rm Card}(A^\oz):=\sup_{f\in F_d}\mathbb
E_\oz\, n(f,\oz):=\sup\limits_{f\in F_d}\int_{\mathcal{F}}
n(f,\omega){\rm d}\mathcal P(\omega).$$ Denote by
$\mathcal{A}_n^{\rm ran}$  the class of all randomized algorithms
with cardinality not exceeding $n$. Then, for a positive integer
$n$, we define
\begin{itemize}
  \item the \emph{randomized case error} of a randomized algorithm $(A^\omega)$ on $F_d$  by
$$e^{\rm ran}({\rm Int}_{\Oz^d,w};F_d,(A^\omega)):=\sup\limits_{f\in F_d}\mathbb E_\omega|{\rm Int}_{\Oz^d,w}(f)-A^{\omega}(f)|;$$
  \item the \emph{$n$-th minimal  (optimal)  randomized case error} on $F_d$  by
\begin{equation*}
  e_n^{\rm ran}({\rm Int}_{\Oz^d,w};F_d):=\inf\limits_{(A^\omega)\in\mathcal{A}_n^{\rm ran}}e^{\rm ran}({\rm Int}_{\Oz^d,w};F_d,(A^\omega)).
\end{equation*}
\end{itemize}

Since a deterministic algorithm may be viewed as a randomized
algorithm, it follows from definitions that
\begin{equation*}
  e_n^{\rm ran}({\rm Int}_{\Oz^d,w};F_d)\le e_n^{\rm det}({\rm Int}_{\Oz^d,w};F_d).
\end{equation*}
Numerous results indicate that  the  optimal randomized quadrature
errors are generally less than  optimal deterministic quadrature
errors. Consequently, employing randomized algorithms proves to be
an effective strategy for reducing errors in a variety of
applications, particularly in dealing with high-dimensional
scenarios.

 Researchers have shown significant interest in the integration problem described in  \eqref{1.1} for Sobolev-type and Besov-type spaces.
 These function spaces are crucial in functional analysis, numerical analysis,
 approximation theory, and related fields (see, for example, \cite{Tem1993, Tem2018, DT}).
We can now review some previous findings in this area.
There is a vast literature on this classical problem for various unweighted function classes, especially for  Sobolev and Besov classes.  For abbreviation, when $w\equiv1$, we  write $ e_n^{\rm det}(F_d)\equiv  e_n^{\rm det}({\rm Int}_{\Omega^d,w};F_d)$ and $ e_n^{\rm ran}(F_d)\equiv  e_n^{\rm ran}({\rm Int}_{\Omega^d,w};F_d)$.
\begin{enumerate}
\item Consider the classical Sobolev class $BW_{p}^{r}([0,1]^{d})$,
$1\le p\le\infty$, $r\in\mathbb N$, defined by
 $$BW_{p}^{r}([0,1]^{d}):=\left\{f\in L_{p}([0,1]^{d}):\, \sum_{|\boldsymbol\alpha|_{1}\le r}\|D^{\boldsymbol\alpha}f\|_{p}\leq1\right\},$$
and  the H\"older class $C^{k,\gamma}([0,1]^d)$, $k\in\mathbb N_0$,
$0<\gz\le1$, defined by
$$C^{k,\gamma}([0,1]^d):=\left\{f\in C([0,1]^d):\,|D^{\boldsymbol\alpha} f({\bf x})-D^{\boldsymbol\az} f({\bf y})|\le \max_{1\le i\le d}|x_i-y_i|^\gz, \, |\boldsymbol\alpha|= k\right\},$$
where $\boldsymbol\alpha\in \mathbb N_0^d,\
|\boldsymbol\alpha|:=\sum_{i=1}^d\alpha_i$, and $D^{\boldsymbol\alpha} f$ is the
partial derivative  of order $\boldsymbol\alpha$ of $f$ in the sense of
distribution. Bakhvalov in \cite{B1} and \cite{B2} proved that\footnote{The notation $A_n\asymp B_n$ means that
$A_n\lesssim B_n$ and $A_n\gtrsim B_n$, and $A_n\lesssim B_n$ ($A_n\gtrsim B_n$)
means that there exists a constant $c>0$ independent of $n$
such that $A_n\leq cB_n$ ($A_n\geq cB_n$).}
 \begin{equation*}e_{n}^{\rm det}(C^{k,\gz}([0,1]^d))\asymp n^{-\frac{k+\gz}{d}}\quad {\rm and}\quad e_{n}^{\rm det}(BW_{\infty}^{r}([0,1]^{d}))\asymp n^{-\frac{r}{d}}.\end{equation*}
 Novak extended the second equivalence  result in \cite{N2} and \cite{N3}, and obtained, for $1\leq p<\infty$  and $r>d/p$,
 \begin{equation*}e_{n}^{\rm det}(BW_{p}^{r}([0,1]^{d}))\asymp n^{-\frac{r}{d}}. \end{equation*}
 Meanwhile, Novak considered the randomized case errors of the above two classes in \cite{N2} and \cite{N3},  and proved that
 \begin{equation*}e_{n}^{\rm ran}(C^{k,\az}([0,1]^d))\asymp
n^{-\frac{k+\az}{d}-\frac{1}{2}},\end{equation*}
and for $1\leq p\le\infty$ and $r>d/p$,
\begin{equation*}e_{n}^{\rm ran}(BW_{p}^{r}([0,1]^{d}))\asymp n^{-\frac{r}{d}-\frac{1}{2}+(\frac{1}{p}-\frac{1}{2})_{+}},\end{equation*}where
$a_{+}=\max(a,0)$.

\item Consider the anisotropic Sobolev class
$BW_{p}^{{\bf r}}([0,1]^{d}),\ 1\leq p\leq\infty,\ {\bf
r}=(r_{1},\cdots, r_{d})\in \mathbb{N}^{d}$, defined by
$$ BW_{p}^{{\bf r}}([0,1]^{d}):=\left\{f\in
L_{p}([0,1]^{d}):\,\sum_{j=1}^{d}\left\|\frac{\partial^{r_{j}}f}{\partial
x_{j}^{r_{j}}}\right\|_p\le 1\right\}.$$  Fang and Ye in
\cite{FY1} obtained for $g({\bf
r}):=\big(\sum_{j=1}^{d}r_{j}^{-1}\big)^{-1}>1/{p}$,
\begin{equation*}e_{n}^{\rm det}(BW_{p}^{{\bf r}}([0,1]^{d}))\asymp n^{-g({\bf r})},\end{equation*} and for $g({\bf
r})>1/{p}$,
\begin{equation*}e_{n}^{\rm ran}(BW_{p}^{{\bf r}}([0,1]^{d}))\asymp n^{-g({\bf
r})-\frac{1}{2}+(\frac{1}{p}-\frac{1}{2})_{+}}.
\end{equation*} For the anisotropic
H\"older-Nikolskii  classes, Fang and Ye  obtained the similar
results in \cite{FY1}.

\item Consider the Sobolev class with bounded mixed derivative
$BW_{p}^{r,\rm mix}([0,1]^{d}),\ r\in\mathbb {N}, \ 1\leq
p\leq\infty$, defined by
$$BW_{p}^{r,\rm mix}([0,1]^{d}):=\left\{f\in
L_{p}([0,1]^{d}):\,\sum_{|\boldsymbol\alpha|_{\infty}\leq
r}\|D^{\boldsymbol\alpha}f\|_{p}\leq1\right\},$$where
$|\boldsymbol\alpha|_{\infty}:=\max_{1\le i\le d}\az_i$. The authors
in \cite{BV, FKK,  SM, T3} obtained for $r>1/{p}$  and  $
1<p<\infty$,
\begin{equation*}e_{n}^{\rm det}(BW_{p}^{r,\rm mix}([0,1]^{d}))\asymp n^{-r}(\log n)^{\frac{d-1}{2}} .\end{equation*}
It was shown in \cite{KN, NUU, U} that for $r>\max\{1/p,1/2\}$ and
$1< p<\infty$,
\begin{equation*}e_{n}^{\rm ran}(BW_{p}^{r,\rm mix}([0,1]^{d}))\asymp n^{-r-\frac{1}{2}+(\frac{1}{p}-\frac{1}{2})_+}.\end{equation*}

\item For the Sobolev class $BW_p^r(\mathbb S^d)$, $1\le p\le
\infty$, $r>0$, on the unit sphere in $\mathbb R^{d+1}$,    it was proved in
\cite{BH, H, Wang} that for  $r>d/p$,
\begin{equation}\label{1.3}e_{n}^{\rm det}(BW_{p}^{r}(\mathbb{S}^d))\asymp n^{-\frac{r}d}.\end{equation}
Wang and Zhang in \cite{WZ} obtained  for $r>d/p$,
\begin{equation}\label{1.4}e_{n}^{\rm ran}(BW_{p}^{r}(\mathbb{S}^d))\asymp  n^{-\frac rd-\frac{1}{2}+(\frac{1}{p}-\frac{1}{2})_+}.
\end{equation}
\item For the  Besov class $BB_{\gamma}^\Theta(L_p(\mathbb S^d))$ of generalized smoothness, $1\le p,\gamma\le \infty,$  Duan, Ye and Li in \cite{DY}
obtained
\begin{equation}\label{1.5}e_{n}^{\rm det}(BB_{\gamma}^\Theta(L_p(\mathbb S^d)))\asymp \Theta(n^{-\frac1d}),\end{equation}
and
\begin{equation}\label{1.6}e_{n}^{\rm ran}(BB_{\gamma}^\Theta(L_p(\mathbb S^d)))\asymp \Theta(n^{-\frac1d}) n^{-\frac{1}{2}+(\frac{1}{p}-\frac{1}{2})_+}.
\end{equation}
   We remark that if $\Theta(t)=t^r,r>0$, the space $B_{\gamma}^\Theta(L_p(\mathbb S^d))$ coincides with the usual Besov space $B_{\gamma}^r(L_p(\mathbb S^d))$,
   which was first introduced in \cite{LR}, and \eqref{1.5} was
   obtained in \cite{HMS}.
  \end{enumerate}

There are few related results for weighted function classes. Let $\Omega^{d}$
 denote the unit sphere $\mathbb S^d$ with the usual surface measure ${\rm d}\sigma$ normalized by $\int_{\mathbb S^{d}}{\rm d}\sigma({\bf x})=1$, or
the unit ball $\mathbb B^d$ or the standard simplex
$\mathbb T^d$ with  the usual Lebesgue measure ${\rm d}{\bf x}$ on $\mathbb R^d$ normalized by $\int_{\Oz^d}{\rm d}{\bf x}=1$. See Sections \ref{sect2} and \ref{sect4} for related definitions.
\begin{enumerate}
  \item[6.] For the weighted Sobolev class $BW_{p,w_\mu}^r(\mathbb B^d)$, $1\le p\le
\infty$, $r>0$, $\mu\ge0$, Li and Wang in \cite{LW} obtained for $r>(d+2\mu)/p$
  \begin{equation}\label{1.9}
  e_{n}^{\rm det}\big({\rm Int}_{\mathbb B^d,w_\mu};BW_{p,w_\mu}^r(\mathbb B^d)\big)\asymp n^{-\frac rd}\end{equation}
and
  \begin{equation}\label{1.10}
  e_{n}^{\rm ran}\big({\rm Int}_{\mathbb B^d,w_\mu};BW_{p,w_\mu}^r(\mathbb B^d)\big)\asymp n^{-\frac rd-\frac12+\left(\frac1p-\frac12\right)_+}.
  \end{equation}
\item[7.]  Dai and Wang in \cite{DW} investigated the weighted Besov
class $BB_\gamma^r(L_{p,w}(\Omega^d))$, $1\le p\le \infty$, $0<\gamma\le
\infty$, $r>0$ with an $A_\infty$ weight $w$ on $\Omega
^d$. They obtained for $r>s_w/p$,
\begin{equation}\label{1.7}
e_{n}^{\rm det}({\rm Int}_{\Omega^d,w};BB_\gamma^r(L_{p,w}(\Omega^d)))\asymp
n^{-\frac{r}d}.
\end{equation}Furthermore, the upper estimate also holds for   all doubling weights. Particularly, for the weighted Besov class $BB_{\gamma}^r(L_{p,w_\mu}(\mathbb B^d))$, $1\le p\le \infty, 0<\gamma\le\infty$, with the Jacobi weight $w_\mu({\bf x}):=(1-\|{\bf x}\|^2)^{\mu-1/2},\mu\ge0$, Li and Wang in \cite{LW} obtained  for $r>(d+2\mu)/p$,
  \begin{equation}\label{1.8}
  e_{n}^{\rm ran}\left({\rm Int}_{\mathbb B^d,w_\mu};BB_\gamma^r(L_{p,w_\mu}(\mathbb B^d))\right)\asymp n^{-\frac rd-\frac12+\left(\frac1p-\frac12\right)_+}.
  \end{equation}
\end{enumerate}


The purpose of this paper is to  generalize the above results
 to weighted Sobolev
classes $BW_{p,w_{\boldsymbol \kappa}}^r(\Omega ^d)$ with
   a Dunkl weight on $\Omega^d$ and
    weighted Besov classes $BB_\gamma^\Theta(L_{p,w}(
\Omega^{d}))$  of generalized smoothness. (See  Sections
\ref{sect2} and \ref{sect4} for related definitions.) Weighted
Sobolev spaces with a Dunkl weight play a critical role in the
Dunkl theory (see \cite{DH, DX, DunXu, X05}). Besov spaces of
generalized smoothness have been studied in various domains,
  and this interest has been increasing in recent years. This growing focus can be attributed to two main directions.
   First, there is an emphasis on exploring the embeddings, limiting embeddings, and approximation results within Besov spaces.
    Second, it is also connected with applications in probability theory, stochastic processes, and the analysis
     of trace spaces on fractals, particularly in the study of so-called $h$-sets (see \cite{DT, FL, CK, SW, HM, Ml, M, Br, WT, WW1}).

 More precisely,  in the
deterministic and randomized  case settings, we obtain  sharp
asymptotic orders of optimal quadrature for the weighted Sobolev
classes $BW_{p,w_{\boldsymbol \kappa}}^r(\Omega ^d)$ with
   a Dunkl weight on $\Omega^d$,
   and  for the weighted Besov classes $BB_\gamma^\Theta(L_{p,w}(
\Omega^{d}))$  of generalized smoothness with a weight satisfying
some conditions on $\Omega^d$.  Our results extends
\eqref{1.3}-\eqref{1.8}.

Now let us briefly introduce our results on the unit sphere
$\mathbb S^d$. For our purpose, we shall consider \emph{the
product weights} of the form
\begin{equation}\label{1.11}
w_{\boldsymbol\kappa}({\bf x}):=\prod_{j=1}^m \left|\langle{\bf
x},{\bf v}_j\rangle\right|^{2\kappa_j},\quad{\bf x}\in\mathbb
R^{d+1},
\end{equation}
where
$\boldsymbol\kappa:=(\kappa_1,\dots,\kappa_m)\in[0,+\infty)^m$ and
${\bf v}_j\in\mathbb S^{d},j=1,\dots,m$. Furthermore, if
$\{\pm{\bf v}_j\}_{j=1}^m$ forms a root system of a finite
reflection group $G$ on $\mathbb S^d$, and $w_{\boldsymbol\kappa}$
is invariant under the reflection group $G$ with nonnegative
$\kappa_j$, then $w_{\boldsymbol\kappa}$ is often called \emph{a
Dunkl weight} (see Section 2 for definitions).

The following Theorem \ref{thm1} gives  results in the
deterministic case setting and its proof is standard. The proof of
upper estimates relies on positive quadrature rules; the proof of
lower estimates is based on traditional methods previously used in
\cite{DW}.

\begin{thm}\label{thm1}Let $1\le p\le\infty$, $r>0$, $0<\gamma\le\infty$, and $\Theta_1(t)$, $\Theta(t)=t^r\Theta_1(t)\in\Phi_s^*$.
\begin{itemize}
  \item[(i)] If $w_{\boldsymbol\kappa}$ is a Dunkl weight with the critical index $s_{w_{\boldsymbol\kappa}}$ on $\mathbb S^d$, then for $r>s_{w_{\boldsymbol\kappa}}/p$,
  \begin{equation}\label{1.4-0}
e_n^{\rm det}\left({\rm Int}_{\mathbb S^{d},w_{\boldsymbol\kappa}}; BW_{p,w_{\boldsymbol\kappa}}^r(\mathbb S^d)\right)\asymp n^{-\frac r{d}};
  \end{equation}
  \item [(ii)] If $w$ is an $A_\infty$ weight with the critical index $s_w$ on $\mathbb S^d$, then for $r>s_w/p$,
  \begin{equation*}
e_n^{\rm det}\left({\rm Int}_{\mathbb S^d,w}; BB_{\gamma}^\Theta(L_{p,w}(\mathbb S^d))\right)\asymp \Theta(n^{-\frac1{d}}).
\end{equation*}Moreover, the upper estimates  hold for   all doubling weights.
\end{itemize}
\end{thm}
\begin{rem}\label{rem1} The upper bound of \eqref{1.4-0} can be derived from \eqref{1.7}. In fact,
 since the Dunkl weight $w_{\boldsymbol\kappa}$ is a doubling weight,
 together with the fact that $W_{p,w_{\boldsymbol\kappa}}^r(\mathbb S^d)$ is
 continuously embedded into $B_\infty^r(L_{p,w_{\boldsymbol\kappa}}(\mathbb S^d))$ (see \eqref{2.8-0}), we have, for $1\le p\le\infty$ and $r>s_{w_{\boldsymbol\kappa}}/p$,
$$e_n^{\rm det}\left({\rm Int}_{\mathbb S^{d},w_{\boldsymbol\kappa}}; BW_{p,w_{\boldsymbol\kappa}}^r(\mathbb S^d)\right)\lesssim e_{n}^{\rm det}({\rm Int}_{\mathbb S^d,w_{\boldsymbol\kappa}};BB_\infty^r(L_{p,w_{\boldsymbol\kappa}}(\mathbb S^d)))\lesssim
n^{-\frac{r}d}.$$
\end{rem}

Our main contribution is the following theorem which gives results
in the randomized case setting. To prove the upper estimates, we
develop the weighted least $\ell_p$ approximation, which is a
``good" $L_p$ approximation, and employ a standard Monte Carlo
algorithm to derive an optimal randomized algorithm.
      We remark that the weighted least $\ell_p$ approximation operator is favored over the filtered hyperinterpolation operators discussed in \cite{WS,WW1,LW,WZ},
      since  the latter is not applicable to general doubling weights.  To prove the lower estimates,
      we utilize Novak's technique, as outlined in \cite{LW}, where the primary challenge lies in constructing ``fooling" functions to satisfy the proper conditions.

\begin{thm}\label{thm2} Let $1\le p\le\infty$, $r>0$, $0<\gamma\le\infty$, and $\Theta_1(t)$, $\Theta(t)=t^r\Theta_1(t)\in\Phi_s^*$.
\begin{itemize}
  \item [(i)] If $w_{\boldsymbol\kappa}$ is a Dunkl weight with the critical index $s_{w_{\boldsymbol\kappa}}$ on $\mathbb S^d$, then for $r>s_{w_{\boldsymbol\kappa}}/p$,
 \begin{equation}\label{1.4-1}
e_n^{\rm ran}\left({\rm Int}_{\mathbb S^{d},w_{\boldsymbol\kappa}}; BW_{p,w_{\boldsymbol\kappa}}^r(\mathbb S^d)\right)\asymp n^{-\frac r{d}-\frac12+\left(\frac1p-\frac12\right)_+}.
\end{equation}
  \item [(ii)] If $w_{\boldsymbol\kappa}$ is a product weight with the critical index $s_{w_{\boldsymbol\kappa}}$ on $\mathbb S^d$, then for $r>s_{w_{\boldsymbol\kappa}}/p$,
  \begin{equation*}
e_n^{\rm ran}\left({\rm Int}_{\mathbb S^d,w_{\boldsymbol\kappa}}; BB_{\gamma}^\Theta(L_{p,w_{\boldsymbol\kappa}}(\mathbb S^d))\right)\asymp \Theta(n^{-\frac1{d}})n^{-\frac12+(\frac1p-\frac12)_+}.
\end{equation*}Moreover, the upper estimates  hold for   all doubling weights.
\end{itemize}\end{thm}
\begin{rem}\label{rem2}
For the special case $w_{\boldsymbol\kappa}({\bf
x})=\prod_{j=1}^{d+1}|x_j|^{2\kappa_j}$, it has been given in
\cite[Theorem 1]{HW} that for $1\le p\le\infty$ and $r>0$,
$$d_n\left(BW_{p,w_{\boldsymbol\kappa}}^r(\mathbb S^d),L_{p,w_{\boldsymbol\kappa}}(\mathbb S^d)\right)\asymp  n^{-\frac r{d}},$$
where $d_n$ are the Kolmogorov numbers. Together with  \cite[Corollary 3]{K} we have
$$
e_n^{\rm ran}\left({\rm Int}_{\mathbb S^{d},w_{\boldsymbol\kappa}}; BW_{2,w_{\boldsymbol\kappa}}^r(\mathbb S^d)\right) \lesssim n^{-\frac12}d_{n}\left(BW_{2,w_{\boldsymbol\kappa}}^r(\mathbb S^d),L_{2,w_{\boldsymbol\kappa}}(\mathbb S^d)\right)\asymp  n^{-\frac12-\frac r{d}}.
$$Then, by the fact that $W_{2,w_{\boldsymbol\kappa}}^r(\mathbb S^d)$ is continuously embedded into $W_{p,w_{\boldsymbol\kappa}}^r(\mathbb S^d)$ for $ 2\le p\le\infty$,
we get  that
\begin{align*}
e_n^{\rm ran}\left({\rm Int}_{\mathbb S^{d},w_{\boldsymbol\kappa}}; BW_{p,w_{\boldsymbol\kappa}}^r(\mathbb S^d)\right)&\le e_n^{\rm ran}\left({\rm Int}_{\mathbb S^{d},w_{\boldsymbol\kappa}}; BW_{2,w_{\boldsymbol\kappa}}^r(\mathbb S^d)\right)\lesssim   n^{-\frac12-\frac r{d}},\ 2\le p\le\infty.
\end{align*}which gives the
partial results of \eqref{1.4-1}.
\end{rem}

In the  proof of Theorem \ref{thm2},  we get asymptotics of the
(linear) sampling numbers of weighted
 Besov class $B_{\gamma}^\Theta(L_{p,w}(\mathbb S^d))$ with an $A_\infty$ weight in $L_{q.w}(\mathbb S^d)$ for the case of $1\le q\le p\le\infty$.
Let $F_d$ be a class of continuous functions on $\Omega^d$, and
$(X,\|\cdot\|_X)$ be a normed linear space of functions on
$\Omega^d$. For  $n\in\mathbb N$, define
\begin{itemize}
\item the \emph{$n$-th sampling number (or optimal recovery)} of $F_d$ in $X$ as
\begin{align*}
g_n(F_d,X) \,:=\,
\inf_{\substack{\ \boldsymbol\xi_1,\dots,\boldsymbol\xi_n\in \Omega^d\\ \varphi:\ \mathbb R^n\rightarrow X}}
\sup_{f\in F_d}\|f-\varphi(f(\boldsymbol\xi_1),\dots,f(\boldsymbol\xi_n))\|_X,
\end{align*}
\item the \emph{$n$-th linear sampling number} of $F_d$ in $X$ as
\begin{align*}
g_n^{\rm lin}(F_d,X) \,:=\,
\inf_{\substack{\boldsymbol\xi_1,\dots,\boldsymbol\xi_n\in \Omega^d\\ \psi_1,\dots,\psi_n\in X}}\,
\sup_{f\in F_d}\,
\Big\|f - \sum_{i=1}^n f(\boldsymbol\xi_i)\, \psi_i\Big\|_X.
\end{align*}
\end{itemize}Clearly, we have
\begin{equation}\label{1.14}
  g_n^{\rm lin}(F_d,X)\ge g_n(F_d,X),
\end{equation} and it is well known (see \cite{TWW}) that for a balanced convex set
$F_d$,
\begin{equation}\label{1.15}
   g_n(F_d,X)\ge\inf_{\substack{\boldsymbol\xi_1,\dots,\boldsymbol\xi_n\in \Omega^d}}\sup_{\substack{f\in F_d\\
  f(\boldsymbol\xi_1)=\dots=f(\boldsymbol\xi_n)=0}}\|f\|_X.
\end{equation}
We obtain the following Theorem \ref{thm1.4} which gives  sharp
asymptotic orders of sampling numbers for $1\le q\le p\le\infty$.
\begin{thm} \label{thm1.4}Let $0<\gamma\le\infty$ and $\Theta_1(t)$, $\Theta(t)=t^r\Theta_1(t)\in\Phi_s^*$.
If $w$ is an $A_\infty$ weight with the critical index $s_w$ on $\mathbb S^d$, and $r>s_{w}/p$, then for $1\le q\le p\le\infty$
\begin{equation}\label{1.16}
   g_n(BB_\gamma^\Theta(L_{p,w}(\mathbb S^d)),L_{q,w}(\mathbb S^d))\asymp \Theta(n^{-\frac1d}).
\end{equation}In particular, for $1\le q\le 2$,
\begin{equation}\label{1.17}
    g_n^{\rm lin}(BB_\gamma^\Theta(L_{2,w}(\mathbb S^d)),L_{q,w}(\mathbb S^d))\asymp \Theta(n^{-\frac1d}).
\end{equation}

Moreover, the above upper estimates hold also for all doubling weights.
\end{thm}
\begin{rem}For the special case $w\equiv 1$, Theorem \ref{thm1.4}  has been given in \cite[Theorem 1.1]{WW1}. We conjecture that for any  doubling weight $w$, and for all $1 \leq p, q \leq \infty$, $r > s_w/p$,
$$g_n(BB_\gamma^\Theta(L_{p,w}(\mathbb S^d)),L_{q,w}(\mathbb S^d))\asymp g_n^{\rm lin}(BB_\gamma^\Theta(L_{p,w}(\mathbb S^d)),L_{q,w}(\mathbb S^d))\asymp \Theta(n^{-\frac1d})n^{\left(\frac1p-\frac1q\right)_+}.$$
\end{rem}

This paper is structured as follows. In Section \ref{sect2}, we
will review some preliminary knowledge on the sphere $\mathbb
S^d$. Section \ref{sect3} will identify  asymptotic orders of the
weighted Sobolev and weighted Besov-type classes on  $\mathbb
S^d$, addressing both deterministic and randomized cases.
Additionally, through the connections between the sphere, the
ball, and the simplex, we will derive corresponding results for
the unit ball and the standard simplex in $\mathbb{R}^d$ in
Section \ref{sect4}. Finally we give some concluding remarks in
Section \ref{sect5}.

\section{Harmonic analysis on the unit sphere}\label{sect2}

In this section we collect some technical facts regarding doubling weights, spherical harmonics, weighted function spaces. Most of materials can be found in \cite{DW,DX,X05}.

\subsection{Basics.}
 For an integer $d\ge1$, we denote by $$\mathbb{S}^{d}:=\{{\bf x}\in \mathbb{R}^{d+1}:\,\|{\bf x}\|^2:=\langle{\bf x},{\bf x}\rangle=x_1^2+\dots+x_{d+1}^2=1\}$$
the unit sphere of $\mathbb R^{d+1}$ endowed with the usual Lebesgue measure ${\rm d}\sigma$ normalized by $\int_{\mathbb S^{d}}{\rm d}\sigma({\bf x})=1$. Let
$d_{\mathbb S}({\bf x},{\bf y}):=\arccos\langle{\bf x}, {\bf y}\rangle$ be the geodesic distance of ${\bf x},{\bf y}\in\mathbb S^{d}$,
and ${\rm c}({\bf x},r):=\{{\bf y}\in \mathbb S^{d}:\, { d}_{\mathbb S}({\bf x},{\bf y})\le r\}$ be the spherical cap centered at ${\bf x}\in\mathbb S^{d}$ with radius $r>0$.
Given a spherical cap $B:={\rm c}({\bf x},r)$ and a constant $c>0$, we denote by $cB$ the spherical cap ${\rm c}({\bf x},cr)$, which has the same center as $B$ but $c$ times the radius. For $\varepsilon>0$, we say a finite subset $\Xi$ of $\mathbb S^d$ is \emph{maximal $\varepsilon$-separated} if
$d_{\mathbb S}({\bf x},{\bf y})>\varepsilon$ for any distinct ${\bf x},{\bf y}\in\Xi$ and $\max_{{\bf x}\in\mathbb S^d}{\rm dist}({\bf x},\Xi)\le\varepsilon$, where ${\rm dist}({\bf x},\Xi):=\min_{{\bf y}\in\Xi}d_{\mathbb S}({\bf x},{\bf y})$.

Given a weight function $w$ (i.e., a nonnegative integrable function) on $\mathbb S^{d}$,  we denote by $L_{p,w}\equiv L_{p,w}(\mathbb S^d)$, $1\le p<\infty$
the space of all real functions $f$ on $\mathbb S^{d}$ endowed with finite norm
$$\|f\|_{p,w}\equiv\|f\|_{L_{p,w}(\mathbb S^{d})}:=\Big(\int_{\mathbb S^{d}}|f({\bf x})|^pw({\bf x}){\rm d}\sigma({\bf x})\Big)^{1/p},$$
and  by  $L_\infty\equiv L_{\infty,w}(\mathbb S^d)\equiv C(\mathbb S^{d})$ the space  consisting of all continuous functions $f$ on $\mathbb S^{d}$ endowed with the norm
$$\|f\|_{\infty}\equiv\|f\|_{L_\infty}:=\sup\limits_{{\bf x}\in\mathbb S^{d}}|f({\bf x})|.$$
In particular,  $L_{2,w}$ is a Hilbert space with the  inner
product
$$\langle f,g\rangle_{L_{2,w}}=\int_{\mathbb S^d}f({\bf x})g({\bf x})w({\bf x}){\rm d}\sigma({\bf x}).$$

  A nonnegative integrable function $w$ on $\mathbb S^{d}$ is called a \emph{doubling weight} if there exists a constant
$L>0$, depending only on $d$ and $w$, such that
$$w(2B)\leq L w(B),\ {\rm\ for\ all\ sperical\ caps}\ B\subset\mathbb S^{d},$$
where $w(E):=\int_{E}w({\bf x}){\rm d}\sigma({\bf x})$ for a measurable subset $E$ of $\mathbb S^{d}$, and the least constant $L$, denoted by $L_w$, is called the doubling constant of $w$.

 A weight function $w$ on $\mathbb S^{d}$ is called an \emph{$A_\infty$ weight} if there exists a constant
$\beta>0$ such that for every spherical cap $B\subset\mathbb S^{d}$ and every measurable $E\subset B$,
$$
\frac{w(B)}{w(E)}\le \beta\,\Big(\frac{ |B|}{ |E|}\Big)^\beta,
$$
where the least constant $\beta$, denoted by $A_\infty(w)$,  is
called the \emph{$A_\infty$ constant} of $w$, and
$|E|:=\int_{E}{\rm d}\sigma({\bf x})$. Obviously, an $A_\infty$
weight must be a doubling weight. In what follows,  we shall use
the symbol $s_w$ to denote a positive number  such that
\begin{equation}\label{2.00}
  \sup\limits_{B\subset\mathbb S^{d}}\frac{w(2^mB)}{w(B)}\le C_{L_w}2^{ms_w},\ m=0,1,2,\dots,
\end{equation}
where $C_{L_w}$ is a constant depending only on the doubling
constant $L_w$. Clearly, such a constant $s_w$  exists (for
example, we can take $s_w=\log L_w/\log2$). Note that if $w=1$,
the least constant $s_w$ for which \eqref{2.00} holds is equal to
$d$.

A typical example of $A_\infty$ weights is  \emph{product weight}
given by
\begin{equation}\label{2.0}
w_{\boldsymbol\kappa}({\bf x})\equiv h_{\boldsymbol\kappa}^2({\bf x}):=\prod_{j=1}^m| \langle{\bf x},{\bf v}_j\rangle|^{2\kappa_j},
\end{equation}
where $\boldsymbol\kappa=(\kappa_1,\dots,\kappa_m)\in[0,+\infty)^m$ and ${\bf v}_j\in\mathbb S^d,j=1,\dots,m$. It was proved in \cite[Example 5.1.5]{DX} that
\[w_{\boldsymbol\kappa}({\rm c}({\bf x},r))\asymp r^d\prod_{j=1}^m\left(| \langle{\bf x},{\bf v}_j\rangle|+r\right)^{2\kappa_j},\ {\bf x}\in\mathbb S^d,\ r\in(0,\pi), \]
and if $E$ is a measurable subset of a spherical cap $B$ in $\mathbb S^d$, then
\begin{equation}\label{Ainfty}
  c_1\frac{|B|}{ |E|}\le \frac{w_{\boldsymbol\kappa}(B)}{ w_{\boldsymbol\kappa}(E)}
  \le c_2\left(\frac{|B|}{ |E|}\right)^\beta,\quad\beta=1+2|\boldsymbol\kappa|:=1+2\sum_{j=1}^m\kappa_j,
\end{equation}for some positive constants $c_1,c_2$ depending only on $d,m$ and $\boldsymbol\kappa$,
 which implies that the product weight $w_{\boldsymbol\kappa}$ is an $A_\infty$ weight on $\mathbb S^d$.

  For a nonzero vector ${\bf v}\in\mathbb{R}^{d+1}$, the reflection $\sigma_{\bf v}$ along ${\bf v}$ is defined by
$$\sigma_{\bf v}{\bf x}:={\bf x}-2\frac{\langle{\bf x},{\bf v}\rangle}{\langle{\bf v},{\bf v}\rangle}{\bf v},\ {\bf x}\in\mathbb{R}^{d+1}.$$
A root system $\mathcal{R}$ is a finite subset of nonzero vectors
in $\mathbb{R}^{d+1}$ such that ${\bf u}\in \mathcal{R}$ implies
$-{\bf u}\in \mathcal{R}$ and ${\bf u},{\bf v}\in \mathcal{R}$
implies $\sigma_{\bf v}{\bf u}\in \mathcal{R}$. For a fixed ${\bf
v}_0\in\mathbb{R}^{d+1}$ such that $\langle {\bf v},{\bf
v}_0\rangle\neq0$ for all ${\bf v}\in \mathcal{R}$, the set of
positive roots $\mathcal{R}_+$ with respect to ${\bf v}_0$ is
defined by $\mathcal{R}_+=\{{\bf v}\in \mathcal{R}:\langle {\bf
v},{\bf v}_0\rangle>0\}$ and
$\mathcal{R}=\mathcal{R}_+\cup(-\mathcal{R}_+)$. A finite
reflection group $G$ with root system $\mathcal{R}$ on
$\mathbb{R}^{d+1}$ is a subgroup of orthogonal group $O(d+1)$
generated by $\{\sigma_{\bf v}:{\bf v}\in \mathcal{R}\}.$ Let
$\kappa_{\bf v}$ be a nonnegative multiplicity function ${\bf
v}\mapsto\kappa_{\bf v}$ defined on $\mathcal{R}_+$ with the
property that $\kappa_{\bf u}=\kappa_{\bf v}$ whenever
$\sigma_{\bf u}$ is conjugate to $\sigma_{\bf v}$ in $G$, that is,
there is a ${\bf w}$ in  $G$ such that $\sigma_{\bf u}{\bf
w}=\sigma_{\bf v}$. Then ${\bf v}\mapsto\kappa_{\bf v}$ is a
$G$-invariant function.

 If $\mathcal{R}:=\{\pm{\bf v}_j\}_{j=1}^m$ forms a root system of a finite reflection group $G$ on $\mathbb R^{d+1}$
 and the product weight $w_{\boldsymbol\kappa}$ given by \eqref{2.0} is invariant under the group $G$,
 then  $w_{\boldsymbol\kappa}$  is called the \emph{Dunkl weight} on $\mathbb S^d$.
For the special Dunkl weight $w_{\boldsymbol\kappa}({\bf
x})=\prod_{j=1}^{d+1}|x_j|^{2\kappa_j}$, it was proved in
\cite[Example 5.1.4]{DX} that the least constant
$s_{w_{\boldsymbol\kappa}}$ for which \eqref{2.00} holds is given
by
\begin{equation}\label{2.000}
s_{w_{\boldsymbol\kappa}}=d+2|\boldsymbol\kappa|-2\min_{1\le j\le d+1}\kappa_j.
\end{equation}

For a nonnegative integer $n$, let $\Pi_n^d$ be the space of all real algebraic polynomials in $d$ variables of degree at most $n$ and $\mathcal{P}_n^{d}$ be the space of real homogeneous algebraic polynomials in $d$ variables of degree $n$. We denote by $\Pi_n(\mathbb{S}^d)$ the space of all real algebraic polynomials in $d+1$ variables of degree at most $n$ restricted to $\mathbb{S}^d$ and $\mathcal{P}_n(\mathbb{S}^d)$ the space of all real homogeneous algebraic polynomials in $d+1$ variables of degree $n$ restricted to $\mathbb{S}^d$.
In this paper, we will need the following weighted Nikolskii's inequality.

\begin{lem}{\rm (\cite[Theorem 5.1]{DX}).} Let $w$ be a doubling weight on $\mathbb S^d$ and $1\le p,q\le\infty$. Then for $P\in\Pi_n(\mathbb S^d)$,
\begin{equation}\label{2.4}
\|P\|_{q,w}\lesssim n^{(\frac1p-\frac1q)_+s_w}\|P\|_{p,w}.
\end{equation}
\end{lem}

\subsection{Weighted Besov classes of generalized smoothness on the sphere.}

Now we give the definition of weighted Besov classes  of generalized smoothness on the  sphere. For the unweighted case, it was first introduced in \cite{WT}.

Let $s\ge1$ be a fixed positive number and let $\Theta$ denote a nonnegative function on $[0,\infty)$.
We say that $\Theta(t)\in\Phi_s^*$
if it satisfies:
\begin{enumerate}
  \item $\Theta(0)=0$ and $\Theta(t)>0$ for all $t>0$;
  \item $\Theta(t)$ is continuous on $[0,\infty)$;
  \item $\Theta(t)$ is almost increasing on $[0,\infty)$, i.e., for any $t_1,t_2$ with $0\le t_1\le t_2$ ,
  $$\Theta(t_1)\le C\Theta(t_2),$$ where $C\ge1$ is a constant independent of $t_1$ and $t_2$;
  \item for any $n\in\mathbb N$ and $t>0$, $$\Theta(nt)\le Cn^s\Theta(t),$$ where $C > 0$ is a constant independent of $n$ and $t$;
  \item there exists $\alpha_1>0$ such that $\Theta(t)/t^{\alpha_1}$ is almost increasing on $[0,\infty)$;
  \item there exists $0<\alpha_2<s$ such that $\Theta(t)/t^{\alpha_2}$ is almost decreasing on $[0,\infty)$, i.e., for any $t_1,t_2$ with $0\le t_1\le t_2$,
  $$\Theta(t_1)/t_1^{\alpha_2}\ge C \Theta(t_2)/t_2^{\alpha_2},$$ where $C>0$ is a constant independent of $t_1$ and $t_2$.
\end{enumerate}


The definition of weighted Besov spaces is based on the \emph{best
approximation} of a function $f\in L_{p,w}$ from $\Pi_n(\mathbb
S^{d})$  defined by
$$E_n(f)_{p,w}:=\inf\limits_{g\in\Pi_n(\mathbb S^{d})}\|f-g\|_{p,w},\ 1\le p\le\infty.$$

Now let $w$ be a doubling weight on $\mathbb{S}^d$ and  $\Theta(t)\in\Phi_s^*$.
For $1\le p\le\infty$ and $0< \gamma\le\infty$, we define  the \emph{weighted Besov spaces of generalized smoothness} $B_{\gamma}^\Theta(L_{p,w})\equiv B_{\gamma}^\Theta(L_{p,w}(\mathbb S^d))$ to be the space of all real
functions $f$ with finite quasi-norm
$$
 \|f\|_{B_{\gamma}^\Theta(L_{p,w})}:=\|f\|_{p,w}+\left\{\sum_{j=0}^\infty\left(\frac{E_{2^j}(f)_{p,w}}{\Theta(2^{-j})}\right)^\gamma\right\}^{1/\gamma},
$$
while the \emph{weighted Besov class}
$BB_{\gamma}^\Theta(L_{p,w})$ is defined to be the unit ball of
the weighted Besov space $B_{\gamma}^\Theta(L_{p,w})$. A prototype
of functions belonging to $\Phi_s^*$ is
$\Theta(t)=t^r(1+(\ln\frac1t)_+)^{-\beta}$, $0<r<s$,
$\beta\in\mathbb R$, in which case,  the space
$B_{\gamma}^\Theta(L_{p,w})$ is the weighted Besov space
$B_{\gamma}^{r,\beta}(L_{p,w})\equiv B_{\gamma}^\Theta(L_{p,w})$
with logarithmic perturbation. Particularly, if $\beta=0$, then
the space $B_{\gamma}^{r,\beta}(L_{p,w})$ coincides with the usual
weighted Besov space $B_{\gamma}^r(L_{p,w})$.

Clearly,  it follows from the definition of $B_{\gamma}^\Theta(L_{p,w})$ that the Jackson inequality
\begin{equation}\label{2.3}
  E_n(f)_{p,w}\lesssim  \Theta(n^{-1})\,\|f\|_{B_{\gamma}^\Theta(L_{p,w})}
\end{equation}holds for $f\in B_{\gamma}^\Theta(L_{p,w})$. Furthermore, we also have the following embedding theorem, which has been proved in \cite[Theorem 2.5]{DW} for the special case of $\Theta(t)=t^r$.

\begin{prop} [Embedding Theorem] \label{thm2.1} Let $\Theta_1(t),\Theta(t):=t^r\Theta_1(t)\in\Phi_s^*$ and let $w$ be a doubling weight on $\mathbb{S}^d$.
 For $1\le p,q\le\infty$ and $0<\gamma\le\infty$, if $r>(1/p-1/q)_+s_w$,
 then\footnote{Here, the notation $X\hookrightarrow Y$ means that the space $X$ is continuously
  embedded into the space $Y$, i.e., $\|f\|_Y\le C\|f\|_X$ for all $f\in X$.}
  $B_\gamma^\Theta(L_{p.w})\hookrightarrow L_{q,w}$. Furthermore, if $r>s_w/p$, then $B_\gamma^\Theta(L_{p.w})\hookrightarrow C(\mathbb S^d)$.
\end{prop}
\begin{proof}
For $\ell=0,1,\dots$, let $\delta_\ell(f)$ be a best approximiant
of $f\in L_{p,w}$ from $\Pi_\ell(\mathbb S^d)$, i.e.,
$E_\ell(f)_{p,w}=\|f-\delta_\ell(f)\|_{p,w}.$ Note that the best
approximant
 polynomials $\delta_\ell(f)$  always exist  since  $\Pi_\ell(\mathbb S^d)$
are the finite dimensional linear spaces (see, for instance,
\cite[p.17, Theorem 1]{L}). Define $$\sigma_j(f)=\delta_{2^{j-1}}(f)-\delta_{2^{j-2}}(f),\ \
{\rm for}\ j=1,2,\dots,$$where $\delta_{2^{-1}}(f)=0$. Obviously, $\sigma_j(f)\in\Pi_{2^{j-1}}(\mathbb S^{d})$. Then it follows from \eqref{2.3} that
\begin{align}
  \|\sigma_j(f)\|_{p,w}&\le\|f-\delta_{2^{j-1}}(f)\|_{p,w}+\|f-\delta_{2^{j-2}}(f)\|_{p,w}\nonumber
\lesssim E_{2^{j-2}}(f)_{p,w}\\&\lesssim\Theta(2^{-(j-2)})\|f\|_{B_\gamma^\Theta(L_{p,w})}.\label{3.6-0}
\end{align}which, by the weighted Nikolskii's inequality \eqref{2.4}, implies that
\begin{align}
  \|\sigma_j(f)\|_{q,w}&\lesssim2^{(j-1)(\frac1p-\frac1q)_+s_w}\Theta(2^{-(j-2)})\|f\|_{B_\gamma^\Theta(L_{p,w})}.\label{3.6-1}
\end{align}By Lusin's theorem and the Stone-Weierstrass theorem,  the space of spherical polynomials is dense in $L_{p,w}$ for
$1\le p<\infty$ and in $C(\mathbb S^d)$ when $p=\infty$, which
implies that the series $\sum_{j=1}^\infty \sigma_j(f)$ converges
to $f$ in the $L_{p,w}$ norm.

For $1\le q<\infty$, by the trigonometric  inequality we obtain
\begin{align*}
  \|f\|_{q,w}&\le \sum_{j=1}^\infty\|\sigma_j(f)\|_{q,w}
 \\& \lesssim\|f\|_{B_\gamma^\Theta(L_{p,w})}\sum_{j=1}^\infty\left[2^{(j-1)(\frac1p-\frac1q)_+s_w}\Theta(2^{-(j-2)})\right]
 \\& \lesssim\|f\|_{B_\gamma^\Theta(L_{p,w})}\sum_{j=1}^\infty2^{-(j-1)[r-(\frac1p-\frac1q)_+s_w]}\Theta_1(2^{-1})
 \\&\lesssim\|f\|_{B_\gamma^\Theta(L_{p,w})},
\end{align*}where  we used $\Theta_1(2^{-(j-2)})\le C\Theta_1(2^{-1})$ for $j=1,2,\dots$ since $\Theta_1\in\Phi_s^*$.

For $q=\infty$, the above argument shows that the series
$\sum_{j=1}^\infty \sigma_j(f)$ converges
uniformly on $\mathbb S^d$. Hence, in this case, $B_\gamma^\Theta(L_{p,w})\hookrightarrow C(\mathbb S^d)$.

 This completes the proof of Proposition \ref{thm2.1}.
\end{proof}

\subsection{Weighted Sobolev classes on the sphere.}

To give the definition of weighted Sobolev spaces on the  sphere,
we shall need some facts about Dunkl harmonics. Now let
$w_{\boldsymbol\kappa}$ be the Dunkl weight of the form
\eqref{2.0} being invariant under the reflection group $G$. The
essential ingredient of the theory of Dunkl harmonics is a family of
first-order differential-difference operators, called Dunkl
operators, which generates a commutative algebra. These operators
are defined by
$$\mathcal{D}_if({\bf x}):=\partial_if({\bf x})+\sum_{j=1}^m\kappa_j\frac{f({\bf x})-f(\sigma_{{\bf v}_j}{\bf x})}{\langle{\bf x},{\bf v}_j\rangle}\langle{\bf v}_j,{\bf e}_i\rangle,\ i=1,\dots,d+1,$$
where $\partial_i\equiv\partial/\partial{x_i}$ denotes the partial derivative with respect to $x_i$, and
 $${\bf e}_1=(1,0,\dots,0),\ {\bf e}_2=(0,1,\dots,0),\dots,\ {\bf e}_{d+1}=(0,0,\dots,1)$$ denote the orthonormal basis
for $\mathbb{R}^{d+1}$. Then the Dunkl-Laplace operator
$\Delta_{\boldsymbol\kappa}$ is defined by
$$\Delta_{\boldsymbol\kappa}:=\mathcal{D}_1^2+\dots+\mathcal{D}_{d+1}^2,$$which plays the role similar to that of the ordinary
Laplacian. It was proved in \cite{DX} that
$$\Delta_{\boldsymbol\kappa}=\frac{\Delta(fw_{\boldsymbol\kappa})-f\Delta(w_{\boldsymbol\kappa})}{w_{\boldsymbol\kappa}}-2\sum_{j=1}^m\kappa_j\frac{f({\bf x})-f(\sigma_{{\bf v}_j}{\bf x})}{\langle{\bf x},{\bf v}_j\rangle^2}.$$
 In particular, if $\mathcal{R}_+=\{{\bf e}_j\}_{j=1}^{d+1}$, $G=\Bbb Z_2^{d+1}$, then
$$\Delta_{\boldsymbol\kappa}f=\Delta f+\sum_{j=1}^{d+1}\kappa_j\left(\frac2{x_j}\partial_jf-\frac{f(x_1,\dots,x_j,\dots,x_{d+1})-f(x_1,\dots,-x_j,\dots,x_{d+1})}{x_j^2}\right).$$

A Dunkl-harmonic polynomial $Y$ is a homogeneous
polynomial $Y$ such that $\Delta_{\boldsymbol\kappa}Y=0$. Furthermore, let
$$\mathcal{H}_n^d(w_{\boldsymbol\kappa}):=\{P\in\mathcal{P}_n^{d+1}:\,\Delta_{\boldsymbol\kappa} P=0\}$$
be the space of the Dunkl harmonics of degree $n$. Then for $n\neq m$,
$$\langle P,Q\rangle_{L_{2,w_{\boldsymbol\kappa}}}=0,\ \ P\in\mathcal{H}_n^d(w_{\boldsymbol\kappa}),\ Q\in \mathcal{H}_m^d(w_{\boldsymbol\kappa}).$$ Throughout this paper, we fix
the value of $\lambda_{\boldsymbol\kappa}$ as
$$\lambda_{\boldsymbol\kappa}:=|\boldsymbol\kappa|+\frac{d-1}2.$$
In terms of the polar coordinates ${\bf y}=r{\bf y}', r=\|{\bf y}\|$, the Dunkl-Laplacian operator $\Delta_{\boldsymbol\kappa}$
takes the form $$\Delta_{\boldsymbol\kappa}:=\frac{\partial^2}{\partial r^2}+\frac{2\lambda_{\boldsymbol\kappa}+1}r\frac{\partial}{\partial r}+\frac1{r^2}\Delta_{{\boldsymbol\kappa},0},$$
where $\Delta_{{\boldsymbol\kappa},0}$ is the Dunkl-Laplace-Beltrami operator on the sphere. Hence, applying $\Delta_{\boldsymbol\kappa}$
to Dunkl harmonics $Y\in\mathcal{H}_n^d(w_{\boldsymbol\kappa})$ with $Y({\bf y})=r^nY({\bf y}')$ shows that spherical Dunkl harmonics
are eigenfunctions of $\Delta_{{\boldsymbol\kappa},0}$; that is,
$$
  \Delta_{{\boldsymbol\kappa},0}Y({\bf x})=-n(n+2\lambda_{\boldsymbol\kappa})Y({\bf x}),\ {\bf x}\in\mathbb{S}^d,\ Y\in \mathcal{H}_n^d(w_{\boldsymbol\kappa}).
$$
The density of polynomial spaces shows that
$$L_{2,w_{\boldsymbol\kappa}}=\bigoplus\limits_{n=0}^\infty\,\mathcal{H}_n^d(w_{\boldsymbol\kappa})\ \ {\rm and}\ \ \Pi_n(\mathbb S^d)=\bigoplus\limits_{k=0}^n\,\mathcal{H}_k^d(w_{\boldsymbol\kappa}).$$
Let
$\{Y_{n,k}:\, k=1,\dots, N(d,n)\}$
be a collection of $L_2$ orthonormal real basis of $\mathcal{H}_n^d(w_{\boldsymbol\kappa})$ (see, e.g., \cite[Chapter 6]{DX}).
Then the orthogonal projector ${\rm proj}_n^{\boldsymbol\kappa} :L_{2,w_{\boldsymbol\kappa}}\rightarrow \mathcal{H}_n^d(w_{\boldsymbol\kappa})$ given by
$${\rm proj}_n^{\boldsymbol\kappa}(f)={\sum_{k=1}^{N(d,n)} \langle f,Y_{n,k}\rangle_{L_{2,w_{\boldsymbol\kappa}}} Y_{n,k}}$$
can be written
as$$
     {\rm proj}_n^{\boldsymbol\kappa}(f,{\bf x})=\int_{\mathbb{S}^d}f({\bf y})P_n(w_{\boldsymbol\kappa};{\bf x},{\bf y})w_{\boldsymbol\kappa}({\bf y}){\rm d}\sigma({\bf y}),
   $$where $P_n(w_{\boldsymbol\kappa};{\bf x},{\bf y})$ is the reproducing kernel of the space of Dunkl harmonics $\mathcal{H}_n^d(w_{\boldsymbol\kappa})$.
 Moreover, for $f\in L_{2,w_{\boldsymbol\kappa}}$,
 $$f({\bf x})=\sum_{n=0}^\infty{\rm proj}_n^{\boldsymbol\kappa}(f,{\bf x}),$$in the $L_{2,w_{\boldsymbol\kappa}}$ norm.

Given $r>0$ and $1\le p\le\infty$, we define the \emph{weighted
Sobolev space} $W_{p,w_{\boldsymbol\kappa}}^r\equiv W_{p,w_{\boldsymbol\kappa}}^r(\mathbb{S}^{d})$ to be the space of all real
functions $f$ with finite norm
$$\|f\|_{W_{p,w_{\boldsymbol\kappa}}^r}:=\|f\|_{p,w_{\boldsymbol\kappa}}+\|(-\Delta_{{\boldsymbol\kappa},0})^{r/2}f\|_{p,w_{\boldsymbol\kappa}},$$
 where the operator $(-\Delta_{{\boldsymbol\kappa},0})^{r/2}$ is given by
$$(-\Delta_{{\boldsymbol\kappa},0})^{r/2}f({\bf x})\ \sim\ \sum_{n=0}^\infty\left(n(n+2\lambda_{\boldsymbol\kappa})\right)^{r/2}{\rm proj}_n^{\boldsymbol\kappa}\left(f,{\bf x}\right),$$
in the distribution sense, while the \emph{weighted Sobolev class}
$BW_{p,w_{\boldsymbol\kappa}}^r$ is defined to be the unit ball of
the weighted Sobolev space $W_{p,w_{\boldsymbol\kappa}}^r$. When
$w_{\boldsymbol\kappa}=1$, $W_{p,w_{\boldsymbol\kappa}}^r$ recedes
to the classical unweighted Sobolev space $W_p^r$. In this case,
we write $W_p^r\equiv W_{p,w_{\boldsymbol\kappa}}^r$ and
$\|\cdot\|_{W_p^r}\equiv
\|\cdot\|_{W_{p,w_{\boldsymbol\kappa}}^r}$ for simplicity.
Moreover, we have the following results:
\begin{itemize}
  \item Bernstein's inequality (\cite[(3.4)]{X05}): for $r>0$, $1\le p\le\infty$, and $P\in\Pi_n(\mathbb S^d)$,
\begin{equation}\label{2.1}
\|(-\Delta_{{\boldsymbol\kappa},0})^{r/2}P\|_{p,w_{\boldsymbol\kappa}}\lesssim n^r\|P\|_{p,w_{\boldsymbol\kappa}};
\end{equation}
  \item Jackson's inequality (\cite[Theorem 3.3]{X05}): for $r>0$, $1\le p\le\infty$, and $f\in W_{p,w_{\boldsymbol\kappa}}^r$,
  \begin{equation}\label{2.2}
  E_n(f)_{p,w_{\boldsymbol\kappa}}\lesssim n^{-r} \|f\|_{W_{p,w_{\boldsymbol\kappa}}^r}.
\end{equation}

 By   Jackson's inequality we obtain
 \begin{equation}\label{2.8-0}
 \|f\|_{B_\infty^r(L_{p,w_{\boldsymbol\kappa}})}:=\|f\|_{p,w_{\boldsymbol\kappa}}+\sup_{j\ge0}2^{jr}E_{2^j}(f)_{p,w_{\boldsymbol\kappa}}\lesssim  \|f\|_{W_{p,w_{\boldsymbol\kappa}}^r},
\end{equation}which means that the weighted Sobolev space
$W_{p,w_{\boldsymbol\kappa}}^r$ can be continuously embedded into
the weighted Besov space
$B_\infty^r(L_{p,w_{\boldsymbol\kappa}})$;

 \item Embedding Theorem: if
$r>(1/p-1/q)_+s_{w_{\boldsymbol\kappa}}$, then
$W_{p,w_{\boldsymbol\kappa}}^r\hookrightarrow
L_{q,w_{\boldsymbol\kappa}}$. Furthermore, if
$r>s_{w_{\boldsymbol\kappa}}/p$, then
$W_{p,w_{\boldsymbol\kappa}}^r\hookrightarrow C(\mathbb S^d)$.
Indeed, this  can be derived from \eqref{2.8-0} and Proposition
\ref{thm2.1}.
\end{itemize}

\section{Proofs of main results}\label{sect3}

In this section we are going to discuss the numerical integration
\eqref{1.1} on the unit sphere. The asymptotic orders of the
weighted Sobolev classes  and the weighted Besov classes in the
deterministic and randomized case settings will be determined. In
the deterministic case setting, for the upper estimates, we use a
positive quadrature formula; for the lower estimates, we will
employ the identity presented in \eqref{1}. In the randomized case
setting, for the upper estimates, we use a constructive polynomial
approximation based on the points that satisfying weighted $L_p$
Marcinkiewicz-Zygmund inequality and the standard Monte-Carlo
methods; for the lower estimates, we will implement Novak's
classical technique involving fooling functions.

\subsection{Deterministic case.}
This subsection is devoted to proving the estimates of deterministic case errors.
More precisely,
let $1\le p\le\infty$ and $0<\gamma\le\infty$. The following results will be  established.
\begin{itemize}
  \item [(i)] If $w_{\boldsymbol\kappa}$ is a Dunkl weight
with the critical index $s_{w_{\boldsymbol\kappa}}$
 on $\mathbb S^d$,
then for $r>s_{w_{\boldsymbol\kappa}}/p$,
\begin{equation}\label{3.1}
e_n^{\rm det}\left({\rm Int}_{\mathbb S^{d},w_{\boldsymbol\kappa}}; BW_{p,w_{\boldsymbol\kappa}}^r\right)\asymp n^{-\frac r{d}}.
\end{equation}
  \item [(ii)] If $w$ is an $A_\infty$ weight
with the critical index $s_w$
 on $\mathbb S^d$, and  $\Theta_1(t)$, $\Theta(t)=t^r\Theta_1(t)\in\Phi_s^*$, then for $r>s_w/p$,
 \begin{equation}\label{3.2}
e_n^{\rm det}\left({\rm Int}_{\mathbb S^d,w};
BB_{\gamma}^\Theta(L_{p,w})\right)\asymp
\Theta\big(n^{-\frac1{d}}\big).
\end{equation}
\end{itemize}

\subsubsection{Upper estimates.}

First we begin to address the proof of the upper estimates. As mentioned in Remark \ref{rem1}, it remains to prove \eqref{3.2}. This can be straightforwardly verified through the following proposition, which is based on the positive cubature formulas.
\begin{prop}\label{prop}Let $1\le p\le\infty$, $0<\gamma\le\infty$, $\Theta_1(t)$, $\Theta(t)=t^r\Theta_1(t)\in\Phi_s^*$, and let $w$ be a doubling weight with the critical index $s_w$ on $\mathbb S^d$. Suppose
that the following positive cubature formula holds for all $f\in\Pi_N(\mathbb S^d)$,
\begin{equation}\label{2}
  \int_{\mathbb S^d}f({\bf x})w({\bf x}){\rm d}\sigma({\bf x})=\sum_{i=1}^n\lambda_i f({\boldsymbol\xi}_i)=:Q_n(f),\ \min_{1\le i\le n}\lambda_i>0,
\end{equation}with $\Lambda:=\{\boldsymbol\xi_i\}_{i=1}^n$ being a finite subset of $\mathbb S^d$ satisfying $N^d\le n\le C_1N^d$. Then, for $r>s_w/p$, we have
\begin{equation}\label{3}
\sup_{f\in BB_\gamma^\Theta(L_{p,w})}\left| {\rm Int}_{\mathbb S^d,w}(f)-Q_n(f)\right|\le C_2\Theta\big(n^{-1/d}\big),
\end{equation}
where the constant $C_2>0$ is independent of $f$ and $n$.
\end{prop}

The proof is based on the following lemma.

\begin{lem}{\rm(\cite[Theorem 3.3]{W1} and \cite[Lemma 5.4.5]{DX}.)}\label{lem3.2}
Let $w$ be a doubling weight on $\mathbb S^d$, and let $\mu$ be a nonnegative measure on $\mathbb S^{d}$ satisfying
\begin{equation}\label{3.10}\int_{\mathbb S^{d}}|f({\bf x} )|^{p_0}{\rm d}\mu({\bf x})\le
C_0\int_{\mathbb S^{d}}|f({\bf x})|^{p_0}w({\bf x}){\rm d}\sigma({\bf x}),\ { for\ all}\ f\in \Pi _N(\mathbb S^{d}),\end{equation}for some $1\le p_0<\infty$. Then for $1\le p<\infty$ and $M\geq N$, we
have
$$\int_{\mathbb S^{d}}|f({\bf x} )|^{p}{\rm d}\mu({\bf x})\le
C_0C\left(\frac MN\right)^{s_w}\int_{\mathbb S^{d}}|f({\bf x})|^{p}w({\bf x}){\rm d}\sigma({\bf x}),\ { for\ all}\ f\in \Pi _M(\mathbb S^{d}),$$where
$C>0$ depends only on $d$, $p$, and $w$.
\end{lem}

Now we are in a position to give the proof of Proposition \ref{prop}.

\begin{proof}[Proof of Proposition \ref{prop}]\

For a fixed positive integer $n$, it is  easy to see that there exists a nonnegative integer $m$ such that
$2^{m-1}\le N<2^m$ and $n\asymp \dim\Pi_N^d\asymp N^d$. We keep the notations $\sigma_j$ and $\delta_k$ in the proof of Proposition \ref{thm2.1}.

According to Theorem \ref{thm2.1} we know that if $f\in B_\gamma^\Theta(L_{p,w})$
and $r>s_w/p$,
then the series $f=\delta_{2^{m-1}}(f)+\sum_{j=m+1}^{\infty}\sigma_j(f)$ converges uniformly on $\mathbb S^d$. Thus, using the
H\"older inequality, we have
\begin{align*}
I_n(f,\Lambda,w)&:=\left| {\rm Int}_{\mathbb S^d,w}(f)-Q_n(f)\right|
\\&=\left|\sum_{j=m+1}^{\infty}\left(\int_{\mathbb S^d}\sigma_j(f)({\bf x})w({\bf x}){\rm d}\sigma({\bf x})-\sum_{i=1}^n\lambda_i \sigma_j(f)({\boldsymbol\xi}_i)\right)\right|\\
&\lesssim\sum_{j=m+1}^{\infty}\left\{\|\sigma_j(f)\|_{p,w}+ \left(\sum_{i=1}^n\lambda_i| \sigma_j(f)({\boldsymbol\xi}_i)|^p\right)^{1/{p}}\right\},
\end{align*}Clearly, we derive from \eqref{2} that \eqref{3.10} is
true for the discrete measure $\mu=\sum_{i=1}^n\lambda_i\delta_{\boldsymbol\xi_i}$ with $p_0=2$ and $C_0=1$. Notice
that $\sigma_j(f)\in\Pi_{2^{j-1}}(\mathbb S^{d})$, and $2^{j-1}\ge N$ for $j\ge m+1$. Then it follows from Lemma \ref{lem3.2} and \eqref{3.6-0} that
\begin{align*}
  \left(\sum_{i=1}^n\lambda_i| \sigma_j(f)({\boldsymbol\xi}_i)|^p\right)^{1/{p}}&\lesssim \left(\frac{2^{j-1}}N\right)^{s_w/p}\|\sigma_j(f)\|_{p,w}
  \\&\lesssim 2^{(j-1)s_w/p}N^{-s_w/p}\Theta(2^{-(j-2)})\|f\|_{B_\gamma^\Theta(L_{p,w})}.
\end{align*}
Thus, we obtain,
\begin{align*}
 I_n(f,\Lambda,w)
&\lesssim N^{-\frac{s_w}p}\sum_{j=m+1}^{\infty}2^{(j-1)s_w/p}\Theta\big(2^{-(j-2)}\big)\|f\|_{B_\gamma^\Theta(L_{p,w})}.
\end{align*}Notice that
\begin{align*}
\sum_{j=m+1}^{\infty}2^{(j-1)s_w/p}\Theta\big(2^{-(j-2)}\big)
&= \sum^{\infty}_{j=m+1}2^{-(j-1)(r-\frac {s_w}p)}\Theta_1\big(2^{-(j-2)}\big)
\\&\lesssim \Theta_1\big(2^{-(m-1)}\big)\sum^{\infty}_{j=m+1}2^{-(j-1)(r-\frac {s_w}p)}
\\&\asymp 2^{-(m-1)(r-\frac {s_w}p)}\Theta_1\big(2^{-(m-1)}\big)
\\&\asymp N^{\frac{s_w}p}\Theta\big(n^{-1/d}\big),
\end{align*}
where in the last  equivalence we used
\begin{equation}\label{3.13-1}
\Theta_1\left(2^{-(m-1)}\right)\asymp\Theta_1\left(n^{-1/d}\right),
\end{equation}
since $\Theta_1\in\Phi_s^*$ and $2^{m-1}\asymp N\asymp n^{-1/d}$.
Therefore, we get
\begin{align*}
 I_n(f,\Lambda,w)
&\lesssim\Theta\big(n^{-1/d}\big)\|f\|_{B_\gamma^\Theta(L_{p,w})},
\end{align*}which proves \eqref{3}.
\end{proof}

\subsubsection{Lower estimates.}

To establish the lower estimates, we require the following lemma.

\begin{lem}\label{lem3.4}{\rm (\cite[Proposition 4.8]{DW}.)}
Let $w$ be an $A_\infty$ weight on $\mathbb S^d$, and let $X$ be a linear subspace of~$\Pi_m(\mathbb S^d)$ with $\dim X\ge\varepsilon\,\dim\Pi_m(\mathbb S^d)$ for some
$\varepsilon\in(0,1)$. Then there exists a function $f\in X$ such
that $\|f\|_{p,w}\asymp 1$ for  $1\le p\le\infty$
with the constants of equivalence depending only on $\varepsilon,d$ and the $A_\infty$ constant of $w$.
\end{lem}

\begin{proof}[Proof of lower estimates.]

\

Let $\boldsymbol\xi_1,\dots,\boldsymbol\xi_n$ be any $n$
distinct points on $\mathbb S^d$. Take a positive integer $N$ such
that $2n\le\dim\Pi_N(\mathbb S^d)\le C_0n$, and denote
\begin{equation}\label{3.5-0}
X_0:=\big\{g\in\Pi_N(\mathbb S^d):\ g(\boldsymbol\xi_j)=0\ {\rm for\ all}\ j=1,\dots,n\big\}.
\end{equation}
Then, $X_0$ is a linear subspace of $\Pi_N(\mathbb S^d)$ with $$\dim X_0\ge\dim\Pi_N(\mathbb S^d)-n\ge\frac12\dim\Pi_N(\mathbb S^d).$$

\emph{Case 1.} Let $w_{\boldsymbol\kappa}$ be a Dunkl weight
 on $\mathbb S^d$. It follows from Lemma \ref{lem3.4} that there exists a function
$g_0\in X_0$ such that
$\|g_0\|_{p_0,w_{\boldsymbol\kappa}}\asymp1$ for all $1\le
p_0\le\infty$. Let $f_0({\bf x})=N^{-r}(g_0({\bf x}))^2$. Then by $f_0\in \Pi_{2N}(\mathbb S^d)$ and  \eqref{2.1}  we have
$$
  \|f_0\|_{W_{p,w_{\boldsymbol\kappa}}^r} =\|f_0\|_{p,w_{\boldsymbol\kappa}}
  +\|(-\Delta_{\boldsymbol\kappa,0})^{r/2}f_0\|_{p,w_{\boldsymbol\kappa}}\lesssim N^r\|f_0\|_{p,w_{\boldsymbol\kappa}}=\|g_0\|_{2p,w_{\boldsymbol\kappa}}^2\asymp
  1.
$$
Hence, there exists a  constant $C>0$ such that $0<f_1=Cf_0\in
BW_{p,w_{\boldsymbol\kappa}}^r$, and
$$f_1(\boldsymbol\xi_1)=\dots=f_1(\boldsymbol\xi_n)=0,\ \
\int_{\mathbb S^d}f_1({\bf x})w_{\boldsymbol\kappa}({\bf x}){\rm
d}\sigma({\bf x})\asymp
N^{-r}\|g_0\|_{2,w_{\boldsymbol\kappa}}^2\asymp n^{-r/d}.$$
 Since $BW_{p,w_{\boldsymbol\kappa}}^r$ is convex and balanced, we derive from \eqref{1} that
\begin{align*}
  e_n^{\rm det}\left({\rm Int}_{\mathbb S^{d},w_{\boldsymbol\kappa}}; BW_{p,w_{\boldsymbol\kappa}}^r\right)\ge
  \inf_{\sub{\boldsymbol\xi_1,\dots,\boldsymbol\xi_n\in \mathbb S^d}}\int_{\mathbb S^d}f_1({\bf x})w_{\boldsymbol\kappa}({\bf x}){\rm d}\sigma({\bf x})\asymp n^{-r/d}.
\end{align*}

 \emph{Case 2.} Let $w$ be an $A_\infty$ weight
 on $\mathbb S^d$, and  $\Theta_1(t)$, $\Theta(t)=t^r\Theta_1(t)\in\Phi_s^*$. It follows from Lemma \ref{lem3.4} that there exists a function $g_0\in\Pi_N(\mathbb S^d)$ with  $\|g_0\|_{p_0,w}\asymp1$ for all $1\le p_0\le\infty$.
Let $f_0({\bf x})=\Theta\left(N^{-1}\right)(g_0({\bf x}))^2$, and
let $s$ be an integer such that $2^{s-1}\le N<2^s$. Then by $f_0\in\Pi_{2^{s+1}}(\mathbb S^d)$ and $E_{2^j}(f_0)_{p,w}\le \|f_0\|_{p,w}$ we have
\begin{align*}
  \|f_0\|_{B_\gamma^\Theta(L_{p,w})} &=\|f_0\|_{p,w}+\left(\sum_{j=0}^{s+1}\Big(\frac{E_{2^j}(f_0)_{p,w}}{\Theta(2^{-j})}\Big)^\gamma\right)^{1/\gamma}
  \\&\lesssim \left(\sum_{j=0}^{s+1}\Big(\frac1{\Theta(2^{-j})}\Big)^\gamma \right)^{1/\gamma}\|f_0\|_{p,w}.
\end{align*}Notice that
\begin{align}\label{3.5-000}
  \sum_{j=0}^{s+1}\left(\frac1{\Theta(2^{-j})}\right)^\gamma&\lesssim 2^{-(s+1)\alpha_1\gamma}\Theta\left(2^{-s-1}\right)^{-\gamma}\sum_{j=0}^{s+1}2^{j\alpha_1\gamma}\notag
  \\&\asymp 2^{-(s+1)\alpha_1\gamma}\Theta\left(2^{-s-1}\right)^{-\gamma}2^{(s+1)\alpha_1\gamma}\notag
  \\&=\Theta\left(2^{-s-1}\right)^{-\gamma},
\end{align}where in the first inequality we used
$$\frac{\Theta\left(2^{-j}\right)}{2^{-j\alpha_1}}\gtrsim \frac{\Theta\left(2^{-(s+1)}\right)}{2^{-(s+1)\alpha_1}},\ 0\le j\le s+1,$$
for some $\alpha_1>0$ since $\Theta\in\Phi_s^*$.
It follows that
\begin{align*}
  \|f_0\|_{B_\gamma^\Theta(L_{p,w})} &\lesssim \Theta\left(2^{-s-1}\right)^{-1}\Theta\left(N^{-1}\right)\|g_0\|_{2p,w}^2\asymp1.
\end{align*}Hence, there exists a  constant $C>0$ such that
$0<f_1=Cf_0\in
BB_\gamma^\Theta(L_{p,w})$, and
$$f_1(\boldsymbol\xi_1)=\dots=f_1(\boldsymbol\xi_n)=0,\  \int_{\mathbb
S^d}f_1({\bf x})w({\bf x}){\rm d}\sigma({\bf x})\asymp \Theta(N^{-1})\|g_0\|_{2,w}^2\asymp \Theta(n^{-1/d}).$$  Since $BB_{\gamma}^\Theta(L_{p,w})$ is convex and balanced, we derive from \eqref{1} that
\begin{align*}
  e_n^{\rm det}\left({\rm Int}_{\mathbb S^d,w}; BB_{\gamma}^\Theta(L_{p,w})\right)&\ge
  \inf_{\sub{\boldsymbol\xi_1,\dots,\boldsymbol\xi_n\in \mathbb S^d}}\int_{\mathbb S^d}f_1({\bf x})w({\bf x}){\rm d}\sigma({\bf x})\asymp \Theta(n^{-1/d}),
\end{align*}
which finishes the proof of the lower estimates.
\end{proof}

\subsection{Randomized case.}

This subsection focuses on establishing estimates for  randomized case errors.
More precisely,
let $1\le p\le\infty$ and $0<\gamma\le\infty$. The following results will be  proved.
\begin{itemize}
  \item [(i)] If $w_{\boldsymbol\kappa}$ is a Dunkl weight with the critical index $s_{w_{\boldsymbol\kappa}}$ on $\mathbb S^d$,
then for $r>s_{w_{\boldsymbol\kappa}}/p$,
\begin{equation}\label{3.23}
e_n^{\rm ran}\left({\rm Int}_{\mathbb S^{d},w_{\boldsymbol\kappa}}; BW_{p,w_{\boldsymbol\kappa}}^r\right)\asymp n^{-\frac r{d}-\frac12+\left(\frac1p-\frac12\right)_+}.
\end{equation}
  \item [(ii)] If $w_{\boldsymbol\kappa}$ is  a product weight  with the critical index $s_{w_{\boldsymbol\kappa}}$ on $\mathbb S^d$, and  $\Theta_1(t)$, $\Theta(t)=t^r\Theta_1(t)\in\Phi_s^*$, then for  $r>s_{w_{\boldsymbol\kappa}}/p$,
  \begin{equation}\label{3.25}
e_n^{\rm ran}\left({\rm Int}_{\mathbb S^d,w_{\boldsymbol\kappa}};
BB_{\gamma}^\Theta(L_{p,w_{\boldsymbol\kappa}})\right)\asymp
\Theta\big(n^{-\frac1{d}}\big)n^{-\frac12+(\frac1p-\frac12)_+}.
\end{equation}
\end{itemize}

\subsubsection{Upper estimates.}

We will first present the proof of upper estimates. As mentioned in Remark \ref{rem1}, it remains to prove the upper estimate of \eqref{3.25}. For our purpose, it is essential to develop an algorithm to attain the upper bound.
We will use constructive polynomial approximation that
uses function values (the samples) at selected well-distributed points (sometimes
called standard information).
Here the ``well-distributed" points indicate that those points satisfy $L_{p,w}$ Marcinkiewicz-Zygmund (MZ) inequalities  on $\mathbb S^d$ as follows.

\begin{defn}Given a set of points $\{{\bf x}_1,{\bf x}_2,\dots,{\bf x}_N\}$
 in $\mathbb S^{d}$ and a set of positive numbers $\{\tau_1,\tau_2,\dots,\tau_N\}$, we say that they constitute
 an \emph{$L_{p,w}$, $1\le p<\infty$, Marcinkiewicz-Zygmund family} for $\Pi_n(\mathbb S^{d})$, denoted by $L_{p,w}$-MZ, if there exist two constants $A,\, B >0$ independent
of $n$ and $N$ such that
\begin{equation}\label{3.5}
  A\|P\|_{p,w}^p\le\sum\limits_{k=1}^{N}\tau_k|P({\bf x}_{k})|^p\le B\|P\|_{p,w}^p,\ {\rm for\ all}\ P\in \Pi_n(\mathbb S^{d});
\end{equation}and they constitute an $L_{\infty}$ Marcinkiewicz-Zygmund family for $\Pi_n(\mathbb S^{d})$, denoted by $L_\infty$-MZ, if there exists a constant $C>0$ independent
of $n$ such that
$$
  \|P\|_{\infty}\le C\max\limits_{1\le k\le N}|P({\bf x}_{k})|,\ {\rm for\ all}\ P\in \Pi_n(\mathbb S^{d}).
$$\end{defn}
\begin{rem}Equivalently, an $L_{p,w}$-MZ family  means that the $L_p$ norm of $P\in\Pi_n(\mathbb S^{d})$ is comparable to the discritized norm $\|P\|_{(p)}$ given by the weighted $\ell_p$ norm of its restriction to these points, where
$$\|f\|_{(p)}:=\Bigg\{
\begin{aligned}&\Big(\sum\limits_{k=1}^{N}\tau_k|f({\bf x}_{k})|^p\Big)^{1/p},&&1\le
p<\infty,
\\
&\max\limits_{k=1,\dots,N}|f({\bf x}_{k})|,
&&p=\infty,\end{aligned}=\Bigg\{
\begin{aligned}&\Big(\int_{\mathbb
S^{d}}|f({\bf x})|^p{\rm d}\mu_N({\bf x})\Big)^{1/p},&&1\le p<\infty,
\\
&\max\limits_{k=1,\dots,N}|f({\bf x}_{k})|,
&&p=\infty,\end{aligned} $$ for $f\in C(\mathbb S^{d})$, and
$\mu_N:=\sum_{k=1}^{N}\tau_{k}\delta_{{\bf x}_{k}}.$
\end{rem}

\begin{rem}\label{rem4}For a doubling weight $w$ on $\mathbb S^d$, such an $L_{p,w}$-MZ family for $\Pi_n(\mathbb S^{d})$ exists. For example,
it follows from \cite[Chapter 5]{DX} that there exists a
$\delta_0>0$ such that for $\delta\in(0,\delta_0)$, a maximal
$\delta/n$-separated subset $\{{\bf x}_1,\dots,{\bf x}_N\}$ is an
$L_{p,w}$-MZ family
 with $\dim\Pi_n(\mathbb S^{d})\le N\asymp n^d$ and $\{\tau_k:=w({\rm c}({\bf x}_k,\delta/n)),k=1,\dots,N\}$.
\end{rem}

\subsubsection*{Weighted least $\ell_p$ approximation.}

Based on an $L_{p,w}$-MZ family, we can construct a weighted least $\ell_p$ approximation on $\mathbb S^{d}$ for a continuous function on $\mathbb S^{d}$.
\begin{defn} Let $1\le p\le\infty$. Then, for $f\in C(\mathbb S^{d})$, its
weighted least $\ell_p$ approximation $L_{n,p}(f)$ is defined by
$$
         L_{n,p}(f):=\arg\min\limits_{P\in\Pi_n(\mathbb S^{d})}\,\|f-P\|_{(p)}.
$$
\end{defn}
\begin{rem} Obviously, the solution of this problem exists for all $f\in
C(\mathbb S^d)$. Further, if $1< p<\infty$, then the solution is
unique. If $p=1$ or $\infty$, then the solution may be not unique
and we can choose any solution to be $ L_{n,p}(f)$.  We note  the
operator $L_{n,p}$ is not linear except the case $p=2$.
\end{rem}
\begin{rem} By the definition of $L_{n,p}$ we have
 \begin{equation}\label{3.6}
 \|f-L_{n,p}(f)\|_{(p)}=\min_{P\in\Pi_n(\mathbb S^{d})}\|f-P\|_{(p)},
 \end{equation}which follows that  for all $P\in\Pi_n(\mathbb S^d)$,
\begin{equation}\label{3.7}
\|f-L_{n,p}(f)\|_{(p)}
  \le \|f-P\|_{(p)}.
\end{equation}
 \end{rem}

Such approximation operators were first introduced by  Gr\"ochenig
in \cite{G} with $p=2$. Given a sequence of Marcinkiewicz-Zygmund
inequalities in $L_2$ on a compact space,
  weighted least squares approximation
 and least squares quadrature were investigated in \cite{G, LuW}. Inspired by these works, for all $1\le p\le\infty$, the authors   in \cite{LLGW} developed  weighted
 least $\ell_p$ approximation induced by a sequence of Marcinkiewicz-Zygmund inequalities
 in $L_p$ on a compact  smooth Riemannian manifold  with  normalized Riemannian measure. In this paper we will use the methods in \cite{LLGW} to study the weighted case on the sphere.
We  now present the approximation theorem for the weighted least
$\ell_p$ approximation.
\begin{thm}\label{thm3.1}  Let $1\le p,q\le\infty$, $0<\gamma\le\infty$, and $\Theta_1(t)$, $\Theta(t)=t^r\Theta_1(t)\in\Phi_s^*$. If $w$ is a doubling weight
   with the critical index $s_w$
   on $\mathbb S^d$, then for $f\in B_\gamma^\Theta(L_{p,w})$ with
$r>s_w/p$,
  \begin{equation}\label{3.9}
  \|f-L_{n,p}(f)\|_{q,w}\le C\Theta\left(n^{-1}\right)n^{\left(\frac1p-\frac1q\right)_+s_w}\|f\|_{B_\gamma^\Theta(L_{p,w})},
\end{equation}
where $C>0$ are independent of  $f$ and $n$.
\end{thm}
\begin{proof} We are going to prove the case of $1\le p<\infty$, since the case $p=\infty$ can be easily checked with a little difference. We keep the notations $\sigma_j$ and $\delta_\ell$ in the proof of Proposition \ref{thm2.1}.
Let $m$ be an integer such that $2^{m-1}\le n<2^m$.
If $r>s_w/p$, then for $1\le q<\infty$,
\begin{align}\label{3.12}
  \|f-L_{n,p}(f)\|_{q,w}&\le\|f-\delta_{2^{m-1}}(f)\|_{q,w}+\|L_{n,p}(f)-\delta_{2^{m-1}}(f)\|_{q,w}\notag
  \\&=:I_1+I_2.
\end{align}

On the one hand, since the
series $\sum_{j=m+1}^\infty\sigma_j(f)$ converges to
$f-\delta_{2^{m-1}}(f)$ in $L_{q,w}$ norm, by \eqref{3.6-1} we obtain
\begin{align*}
  I_1&\lesssim \sum_{j=m+1}^\infty\| \sigma_j(f)\|_{q,w}
   \lesssim \sum\limits_{j=m+1}^\infty 2^{(j-1)\left(\frac1p-\frac1q\right)_+s_w}\Theta(2^{-(j-2)})\,\|f\|_{B_\gamma^\Theta(L_{p,w})},
\end{align*}Notice that
\begin{align}\label{3.13-00}
  &\quad\sum\limits_{j=m+1}^\infty 2^{(j-1)\left(\frac1p-\frac1q\right)_+s_w}\Theta(2^{-(j-2)})\notag
  \\&\lesssim 2^{(m-1)\alpha_2}\Theta_1(2^{-(m-1)})\sum_{j=m+1}^\infty 2^{(j-1)\left(-\alpha_2-r+\left(\frac1p-\frac1q\right)_+s_w\right)}\notag
  \\&\asymp 2^{(m-1)\alpha_2}\Theta_1(2^{-(m-1)}) 2^{(m-1)\left(-\alpha_2-r+\left(\frac1p-\frac1q\right)_+s_w\right)}\notag
  \\&\lesssim \Theta(n^{-1})n^{\left(\frac1p-\frac1q\right)_+s_w},
\end{align}where in the  first inequality we used
$$\frac{\Theta_1\left(2^{-(j-2)}\right)}{2^{-(j-2)\alpha_2}}\lesssim \frac{\Theta_1\left(2^{-(m-1)}\right)}{2^{-(m-1)\alpha_2}},\ j\ge m+1,$$
for some $\alpha_2>0$ since $\Theta_1\in\Phi_s^*$. Thus, we get
\begin{align}\label{3.13}
  I_1\lesssim \Theta(n^{-1})n^{\left(\frac1p-\frac1q\right)_+s_w}\,\|f\|_{B_\gamma^\Theta(L_{p,w})},
\end{align}

On the other hand,   we derive from
\eqref{3.5} and \eqref{3.7} that
\begin{align}
  \|L_{n,p}(f)-\delta_{2^{m-1}}(f)\|_{p,w}&\lesssim\|L_{n,p}(f)-\delta_{2^{m-1}}(f)\|_{(p)}
  \notag\\&\le \|L_{n,p}(f)-f\|_{(p)}+\|f-\delta_{2^{m-1}}(f)\|_{(p)}
  \notag\\&\le 2\|f-\delta_{2^{m-1}}(f)\|_{(p)}
  \nonumber\\&\lesssim\sum\limits_{j=m+1}^\infty\|
  \sigma_j(f)\|_{(p)}.\label{3.14}
\end{align}
According to \eqref{3.5} we know that \eqref{3.10} is
true for $\{\tau_{k}\}_{k=1}^{N}$ with $p_0=p$ and $C_0=B$. Note
that $\sigma_j(f)\in\Pi_{2^{j-1}}(\mathbb S^{d})$, and $2^{j-1}\ge n$ for $j\ge m+1$.
It follows from Lemma \ref{lem3.2}
 that  for $j\ge m+1$,
\begin{align*}
  \| \sigma_j(f)\|_{(p)}^p
  &=\int_{\mathbb S^{d}}|\sigma_j(f)({\bf x})|^p{\rm d}\mu_N({\bf x})
  \\&\lesssim \left(\frac{2^{j-1}}n\right)^{s_w}\|\sigma_j(f)\|_{p,w}^p
  \\&\lesssim
  \left(\frac{2^{j-1}}n\right)^{s_w}\Theta(2^{-(j-2)})^p\,\|f\|_{B_\gamma^\Theta(L_{p,w})}^p
\end{align*}that is,
\begin{equation}\label{3.15}
  \| \sigma_j(f)\|_{(p)}\lesssim n^{-\frac{s_w}p}2^{(j-1)\frac{s_w}p}\Theta(2^{-(j-2)})\,\|f\|_{B_\gamma^\Theta(L_{p,w})}.
\end{equation}
Hence, for $r>s_w/p$, by \eqref{3.15},
\eqref{3.14}, and \eqref{3.13-00},  we obtain
\begin{align}\label{3.16}
  \|L_{n,p}(f)-\delta_{2^{m-1}}(f)\|_{p,w}&\lesssim \sum_{j=m+1}^\infty\| \sigma_j(f)\|_{(p)}\notag
  \\&\lesssim n^{-\frac{s_w}p}
  \sum_{j=m+1}^\infty 2^{(j-1)\frac{s_w}p}\Theta\left(2^{-(j-2)}\right)\|f\|_{B_\gamma^\Theta(L_{p,w})}\notag
  \\&\asymp \Theta\left(n^{-1}\right)\|f\|_{B_\gamma^\Theta(L_{p,w})}.
\end{align}
Note that $L_{n,p}(f)-\delta_{2^{m-1}}(f)\in\Pi_n(\mathbb S^d)$. Then by the weighted Nikolskii's inequality \eqref{2.4}, \eqref{3.16} implies that
\begin{align}\label{3.16-0}
  I_2&\lesssim n^{\left(\frac1p-\frac1q\right)_+s_w}\|L_{n,p}(f)-\delta_{2^{m-1}}(f)\|_{p,w}\notag
  \\&\lesssim \Theta(n^{-1})n^{\left(\frac1p-\frac1q\right)_+s_w}\,\|f\|_{B_\gamma^\Theta(L_{p,w})}.
  \end{align}

By \eqref{3.12}, \eqref{3.13}, and  \eqref{3.16-0}, we get
\eqref{3.9} for $1\le p<\infty$. This completes the proof.
\end{proof}

We observe that the orders specified in Theorem \ref{thm3.1} are optimal for $1\le q\le p\le\infty$. This can be demonstrated through
 the optimal recovery of  weighted generalized Besov space $B_\gamma^\Theta(L_{p,w}(\mathbb S^d))$ as stated in Theorem \ref{thm1.4}.
 Now we give the proof of Theorem \ref{thm1.4}.

\begin{proof}[Proof of Theorem \ref{thm1.4}]\

 For the upper bound, let $L_{m,p}$ be the weighted least $\ell_p$ approximation operator induced by an $L_{p,w}$-MZ family $\{{\bf x}_1,\dots,{\bf x}_N\}$ with $\dim\Pi_m(\mathbb S^d)\le N\asymp m^d\asymp n$. Then, by Theorem \ref{thm3.1} we have
$$g_n(BB_\gamma^\Theta(L_{p,w}),L_{q,w})\le\sup_{f\in BB_\gamma^\Theta(L_{p,w})}\|f-L_{m,p}(f)\|_{q, w}\lesssim \Theta(m^{-1})\asymp \Theta(n^{-1/d}).$$

Now  we turn to show the lower bound.  To this end, let
$\boldsymbol\xi_1,\dots,\boldsymbol\xi_n$ be any $n$ distinct
points on $\mathbb S^d$. Then following the proof of the lower
estimates in the deterministic case setting, we know that there
exists a function $f_1\in BB_\gamma^\Theta(L_{p,w})$ satisfying
$f_1(\boldsymbol\xi_1)=\dots=f_1(\boldsymbol\xi_n)=0$ and
$\|f_1\|_{q,w}\asymp \Theta(n^{-1/d})$. Then it follows from
\eqref{1.15} that
$$g_n(BB_\gamma^\Theta(L_{p,w}),L_{q,w})\ge\inf_{\substack{\boldsymbol\xi_1,\dots,\boldsymbol\xi_n\in
\mathbb S^d}}\|f_1\|_{q,w}\asymp
\Theta(n^{-1/d}).$$

Together with the fact that the weighted least $\ell_2$ approximation operator  is linear and \eqref{1.15}, the proof of Theorem \ref{thm1.4} is finished.
\end{proof}

According to Remark \ref{rem4}, we know that for an $A_\infty$ weight and $1\leq
p\leq\infty$, there exists an $L_{p,w}$-MZ family
on $\mathbb S^d$ with $\dim\Pi_m(\mathbb S^d)\le N\asymp m^{d}\asymp n$. Based on the proof of Theorem
\ref{thm1.4}, we can conculde that for $1\leq
q\le p\leq \infty$ and $r>s_w/{p}$,
$$\sup_{f\in
BB_\gamma^\Theta(L_{p,w})}\|f-L_{m,p}(f)\|_{q,w}\asymp
\Theta\left(n^{-1/d}\right)\asymp g_n(BB_\gamma^\Theta(L_{p,w}),L_{q,w}),$$
 which implies that
the weighted least $\ell_p$  approximation operator
$L_{m,p}$ serves as an
asymptotically optimal algorithm for optimal recovery of
$BB_\gamma^\Theta(L_{p,w})$ measured in the $L_{q,w}$ norm.

\subsubsection*{Standard Monte
Carlo algorithm.} To achieve the upper estimates, following
Heinrich \cite{H1}, we also need a concrete Monte Carlo method by
virtue of the standard Monte Carlo algorithm. It is defined as
follows: let $\{\boldsymbol\xi_i\}_{i=1}^M$ be independent,
$\mathbb S^{d}$-valued, distributed over $\mathbb S^{d}$ with
respect to the measure $w({\bf x}){\rm d}\sigma({\bf x})$ random
vectors on a probability space $(\mathcal{F},\Sigma,\nu)$. For any
$h\in C(\mathbb S^{d})$, we put
\begin{equation}\label{3.20} Q_M^\oz(h)=\frac1M\sum_{i=1}^M
h\left(\boldsymbol\xi_i(\oz)\right),\ \oz\in\mathcal{F}.
\end{equation} Then
$$\mathbb E_\oz Q_M^\oz(h)=\int_{\mathbb S^d}h({\bf x})w({\bf x}){\rm d}\sigma({\bf x}).$$
Following the proof of
\cite[Proposition 5.4]{H1}, we have, for any $h\in C(\mathbb S^{d})$ and $1\le p\le\infty$,
\begin{equation}\label{3.21}\mathbb E_\oz |{\rm Int}_{\mathbb S^{d},w}(h)-Q_M^\oz(h)|\lesssim M^{-\frac12+(\frac1p-\frac12)_+}\|h\|_{p,w}.
\end{equation}Here we omit the proof.

\begin{proof}[Proof of upper estimates of \eqref{3.25}]\

 For a positive integer $n$, define $ M = \lfloor n/2 \rfloor$ and let $\dim \Pi_m(\mathbb{S}^d)\le N \asymp m^d \asymp n$ with $N \le n/2$. Here, $\lfloor x \rfloor$ denotes the largest integer not exceeding $x$. Let $L_{m,p}$ be the weighted least $\ell_p$ approximation operator induced by an $ L_{p,w}$-MZ family $\{{\bf x}_1, \ldots, {\bf x}_N\}$. For a function $f \in C(\mathbb{S}^d)$, we introduce a randomized algorithm $(A_n^{\omega})$ defined as follows:
\[
A_n^{\omega}(f) = Q_M^{\omega}(f - L_{m,p}(f)) + \text{Int}_{\mathbb{S}^{d}, w}(L_{m,p}(f)),
\]
where $Q_M^{\omega}$ is the standard Monte Carlo algorithm  of the
form \eqref{3.20}. It is evident that the algorithm $
(A_n^{\omega})$ utilizes at most $M + N \le n$ values of the
function $f$, thus confirming that $(A_n^{\omega}) \in
\mathcal{A}_n^{\text{ran}}$. Note that the algorithm
$(A_n^{\omega})$ is not linear except $p=2$. A straightforward
verification gives
\[
|{\rm Int}_{\mathbb S^{d},w}(f)-A_n^{\rm \oz}(f)|=|{\rm Int}_{\mathbb S^{d},w}(g)-Q_M^\oz(g)|,
\]
where $g=f-L_{m,p}(f)$.
Combining with \eqref{3.21} and \eqref{3.9}, we obtain, for $f\in
B_\gamma^\Theta(L_{p,w})$, $r>s_{w}/p$,
\begin{align*}\mathbb E_\oz|{\rm Int}_{\mathbb S^{d},w}(f)-A_n^{\oz}(f)|&=\mathbb E_\oz |{\rm Int}_{\mathbb S^{d},w}(g)-Q_M^\oz(g)|\\&\lesssim M^{-\frac12+(\frac1p-\frac12)_+}\|f-L_{m,p}(f)\|_{p,w}
\\&\lesssim M^{-\frac12+(\frac1p-\frac12)_+}\Theta\big(m^{-1}\big)\|f\|_{B_\gamma^\Theta(L_{p,w})}
\\&\asymp \Theta(n^{-\frac1d})n^{-\frac12+(\frac1p-\frac12)_+}\|f\|_{B_\gamma^\Theta(L_{p,w})},
\end{align*}which leads to
$$e_n^{\rm ran}\left({\rm Int}_{\mathbb S^{d},w};BB_\gamma^\Theta(L_{p,w})\right)\lesssim \Theta(n^{-\frac1d})n^{-\frac12+(\frac1p-\frac12)_+}.$$

This completes  the proof of the upper estimate of  \eqref{3.25}.
\end{proof}

We remark that the upper estimate of \eqref{3.25} holds whenever $w$ is a doubling
weight.

\subsubsection{Lower estimates}

For the proof of  lower estimates, we use the technique of Novak,
which relies on the following lemma.
\begin{lem}{\rm(\cite{N1} or \cite[Lemma 3]{N2}.)}\label{lem3.5} Assume that $\Omega^d$ is the multivariate compact domain with a  weight function $w$.
\begin{itemize}
  \item [(a)] Let $F\subset L_{1,w}(\Omega^d)$ and $\psi_j,\,j = 1, \dots, 4n$, with the
following conditions:
\begin{itemize}
  \item[(i)] the functions $\psi_j,\,j = 1, \dots, 4n$ have disjoint supports and satisfy
$${\rm Int}_{\Omega^d,w}(\psi_j)=\int_{\Omega^d} \psi_j({\bf x})w({\bf x}){\rm d}{\bf x}\ge \dz,\ {\rm for}\ j= 1,\dots, 4n.$$

 \item[(ii)] $F_1:=\big\{\sum_{j=1}^{4n}
\az_j\psi_j:\,\alpha_j\in\{-1,1\}\big\}\subset F$.
\end{itemize}
\noindent Then
$$e_n^{\rm ran}({\rm Int}_{\Omega^d,w};F)\ge  \frac{1}{2}\dz n^{\frac12}.$$
  \item [(b)] We assume that instead of (ii) in statement (a) the
property

 (ii') $F_2:=\left\{\pm \psi_j:\, j= 1, \dots,
4n\right\}\subset F$. \vskip 1mm

\noindent Then
$$e^{\rm ran}_n ({\rm Int}_{\Omega^d,w};F)\ge \frac14\dz.$$
\end{itemize}
\end{lem}

Therefore, to obtain the lower estimates of randomized quadrature errors, we shall construct a sequence of functions
$\{\psi_j\}_{j=1}^{4n}$ satisfying the conditions of Lemma
\ref{lem3.5}, which are often call the fooling functions. Their construction follows by the similar arguments as stated in the estimates of the entropy numbers (see \cite{WW}).

\subsubsection*{Construct fooling functions.}

Let
$\boldsymbol\kappa:=(\kappa_1,\dots,\kappa_m)\in[0,+\infty)^m$,
${\bf v}_j\in\mathbb S^{d}$, $j=1,\dots,m$, and
$$w_{\boldsymbol\kappa}({\textbf
x}):=\prod_{j=1}^m\left|\langle{\bf x},{\bf
v}_j\rangle\right|^{2\kappa_j},\ {\bf x}\in\mathbb R^{d+1}.$$ For
$j=1,\dots,m$, we denote
$$E_j:=\left\{{\bf x}\in\mathbb S^{d}:\left|\frac\pi2-d_{\mathbb S}({\bf x},{\bf v}_j)\right|\le2\varepsilon_{d,m}\right\},$$
and
$$\widetilde{E}_j:=\left\{{\bf x}\in\mathbb S^{d}:\left|\frac\pi2-d_{\mathbb S}({\bf x},{\bf v}_j)\right|\le\varepsilon_{d,m}\right\},$$
where $\varepsilon_{d,m}$  is a sufficiently small positive
constant depending only on $d$ and $m$. Furthermore, let
$E:=\bigcup_{j=1}^m E_j$ and $\widetilde{E}=\bigcup_{j=1}^m
\widetilde{E}_j$. Clearly, we have $|E_1|=\dots= |E_m|\asymp
\varepsilon_{d,m}$.  Then the Lebesgue measure of $E$ satisfies
$$|E|\le \sum_{j=1}^m|E_j|\le c_dm\varepsilon_{d,m}\le\frac12 ,$$
provided that $\varepsilon_{d,m}$ is small enough. Thus, the
Lebesgue measure of $\mathbb S^{d}\backslash \widetilde{E}$
satisfies
$$|\mathbb S^{d}\backslash \widetilde{E}|\ge |\mathbb
S^{d}\backslash {E}|\ge\frac12 .$$ Given a positive
integer $n$, we can choose a sufficient large positive integer $N$
such that $N\asymp n^{1/d}$ and $N> \varepsilon_{d,m}^{-1} $. We
take a maximal $2/N$-separated subset $\{{\bf x}_i\}_{i=1}^{M}$ of
$\mathbb S^{d}\backslash {E}$. Then
\begin{equation}\label{0}
{\rm c}\big({\bf x}_i,\frac1N\big)\bigcap{\rm c}\big({\bf x}_j,\frac1N\big)=\emptyset,\ \ {\rm if}\ i\neq j,
\end{equation}which follows from the definition of maximal separated subsets,
and
\begin{equation}\label{00}
{\rm c}\big({\bf x}_j,\frac1N\big)\subset\mathbb S^{d}\backslash
\widetilde{E},\ j=1,\dots,M,
\end{equation}since
\begin{align*}\left|\frac{\pi}2-d_{\mathbb S}({\bf y},{\bf v}_i)\right|&\ge
\left|\frac{\pi}2-d_{\mathbb S}({\bf x}_j,{\bf
v}_i)\right|-\big|d_{\mathbb S}({\bf x}_j,{\bf v}_i)-d_{\mathbb
S}({\bf v}_i,{\bf y})\big|\\ & \ge \left|\frac{\pi}2-d_{\mathbb
S}({\bf x}_j,{\bf
v}_i)\right|-d_{\mathbb S}({\bf x}_j,{\bf y})\\
&>2\varepsilon_{d,m}-\frac1N>\varepsilon_{d,m},\
i=1,\dots,m\end{align*}for any ${\bf y}\in {\rm c}\big({\bf
x}_j,\frac1N\big)$ whenever $N>\varepsilon_{d,m}^{-1}$. Meanwhile,
notice that
$$\frac12\le |\mathbb S^{d}\backslash {E}|\le \left|\bigcup_{i=1}^M{\rm c}\big({\bf x}_i,\frac2N\big)\right|\le M\left|{\rm c}\big({\bf x}_1,\frac2N\big)\right|\asymp MN^{-d}\asymp Mn^{-1}$$
by the maximal separated property of $\{{\bf x}_i\}_{i=1}^{M}$.
Thus, we have  $M\ge 4n$ for sufficiently large $N$ with $N\asymp
n^{1/d}$.


If ${\bf x}\in\mathbb S^{d}\backslash \widetilde E$, then
$$\left|\langle{\bf x},{\bf v}_j\rangle\right|\ge\sin(\varepsilon_{d,m}),\ j=1,\dots,m.$$
Hence,
$$w_{\boldsymbol\kappa}({\bf x})=\prod_{j=1}^m\left|\langle{\bf x},{\bf v}_j\rangle\right|^{2\kappa_j}\ge \left(\sin(\varepsilon_{d,m})\right)^{2|{\boldsymbol\kappa}|},$$
which leads that
$$w_{\boldsymbol\kappa}({\bf x})\asymp 1,\ \ {\rm for\ any}\ {\bf x}\in\mathbb S^{d}\backslash \widetilde E.$$

Let $\varphi$ be a nonnegative $C^\infty$-function on $\mathbb{R}$
supported in $[0,1]$ and being equal to 1 on $[0,1/2]$. As the
above construction, we can choose $4n$ points $\{{\bf
x}_i\}_{i=1}^{4n}$ from a maximal $2/N$-separated subset of
$\mathbb S^{d}\backslash {E}$, and define
$$\varphi_j({\bf x})=\varphi(Nd_{\mathbb S}\left({\bf x},{\bf x}_j)\right),\  j=1,\dots,4n.$$
 Clearly, $\varphi_j,j=1,\dots,4n$ have the following properties:
\begin{itemize}
  \item [(i)] the support sets of $\{\varphi_j\}_{j=1}^{4n}$ are mutually disjoint, that is
  \begin{equation}\label{3.26}{\rm supp}(\varphi_i)\bigcap {\rm supp}(\varphi_j)=\emptyset,\ {\rm for}\ i\neq
j.\end{equation}Indeed, it can be seen from the fact $${\rm supp}(\varphi_j)\subset{\rm c}\big({\bf x}_j,\frac1N\big)
\subset\mathbb S^d\backslash \widetilde{E},\ j=1,\dots,4n.$$

  \item [(ii)] for  $1\le p<\infty$, the $L_{p,w_{\boldsymbol\kappa}}$ norm of $\varphi_j$
  satisfies
  \begin{equation}\label{3.27}\|\varphi_j\|_{p,w_{\boldsymbol\kappa}}\asymp\Big(\int_{{\rm c}({\bf x}_j,\frac1N)}|\varphi(Nd_{\mathbb S}\left({\bf x},{\bf x}_j)\right)|^p{\rm d}\sigma({\bf x})\Big)^{1/p}\asymp  n^{-1/p},
\end{equation}
and the $L_{\infty}$ norm of $\varphi_j$ satisfies
  \begin{equation}\label{3.28}\|\varphi_j\|_{\infty}=\sup_{{\bf x}\in{\rm c}({\bf x}_j,\frac1N)}|\varphi(Nd_{\mathbb S}\left({\bf x},{\bf x}_j)\right)|=\max_{x\in[0,1]}|\varphi(x)|\asymp1.
\end{equation}
\end{itemize}

Now we set
$$F_0:=\Big\{f_{\boldsymbol\az}:=\sum\limits_{j=1}^{4n}
\az_j\varphi_j:\,
{\boldsymbol\az}=(\az_1,\cdots,\az_{4n})\in\mathbb{R}^{4n}\Big\}.$$
Then it follows from \eqref{3.26}, \eqref{3.27}, and \eqref{3.28} that for any $f_{\boldsymbol\az}\in F_0$,
\begin{equation}\label{3.29}\|f_{\boldsymbol\az}\|_{p,w_{\boldsymbol\kappa}}\asymp
 n^{-1/p}\,\|{\boldsymbol\az}\|_{\ell_p^{4n}},\ 1\le p\le \infty, \end{equation}
where $$\|{\boldsymbol\az}\|_{\ell_p^{4n}}:=\Bigg\{
\begin{aligned}&\Big(\sum_{j=1}^{4n}|\az_j|^p\Big)^{1/p},&&1\leq
p<\infty,
\\
&\max\limits_{1\le j\le
4n}|\az_j|,
&&p=\infty.\end{aligned}$$Furthermore, we have the following lemma.

\begin{lem}\label{lem44} Let $1\le p\le\infty$, $0<\gamma\le\infty$, $\Theta_1(t)$, $\Theta(t)=t^r\Theta_1(t)\in\Phi_s^*$, and $f_{\boldsymbol\az}\in F_0$ with $\boldsymbol\az=(\az_1,\dots,\az_{4n})\in\mathbb R^{4n}$.
\begin{itemize}
  \item[(i)] If $w_{\boldsymbol\kappa}$ is a Dunkl weight on $\mathbb S^d$, then for any $r>0$,
\begin{align}\label{3.30}\|f_{{\boldsymbol\az}}\|_{W_{p,w_{\boldsymbol\kappa}}^r}\lesssim
n^{r/d-1/p}\,\|{\boldsymbol\az}\|_{\ell_p^{4n}}.
\end{align}When $w_{\boldsymbol\kappa}=1$, the above inequality
also holds.
  \item[(ii)] If $w_{\boldsymbol\kappa}$ is a product weight on $\mathbb S^d$, then for any $r>0$,
\begin{align}\label{3.31}\|f_{{\boldsymbol\az}}\|_{B_\gamma^\Theta(L_{p,w_{\boldsymbol\kappa}})}\lesssim \Theta(n^{-1/d})^{-1}n^{-1/p}\,\|{\boldsymbol\az}\|_{\ell_p^{4n}},
\end{align}\end{itemize}
\end{lem}
\begin{proof} (i) By \eqref{3.29} it suffices to show that for all $r>0$,
\begin{equation}\label{3.32}
\|\left(-\Delta_{{\boldsymbol\kappa},0}\right)^{r/2}f_{\boldsymbol\alpha}\|_{p,w_{\boldsymbol\kappa}}\le N^{r-d/p}\|{\boldsymbol\az}\|_{\ell_p^{4n}}.
\end{equation}

First we claim that for any positive integer $v$,
\begin{equation}\label{3.33}
\left\| \left(-\Delta_{{\boldsymbol\kappa},0}\right)^vf_{\boldsymbol\az}\right\|_{p,w_{\boldsymbol\kappa}}\lesssim N^{2v-d/p}\|{\boldsymbol\az}\|_{\ell_p^{4n}}.
\end{equation}
Indeed, if $f_{\boldsymbol\az}\in F_0$ then we have for any $v=1,2,\dots$,
\begin{equation}\label{3.34}
  \left(-\Delta_{{\boldsymbol\kappa},0}\right)^vf_{\boldsymbol\az}({\bf x})=\sum_{j=1}^{4n}\alpha_j\left(-\Delta_{{\boldsymbol\kappa},0}\right)^v\varphi_j({\bf x}).
\end{equation}
We will prove \eqref{3.33} by four steps.

Step 1. We shall show that for any $v=1,2,\dots$,
\begin{equation}\label{3.35}
{\rm supp}\left(-\Delta_{{\boldsymbol\kappa},0}\right)^v\varphi_j\subset \bigcup_{\rho\in G}\rho\Big({\rm c}\big({\bf x}_j,\frac1N\big)\Big),\ j=1,\dots, 4n,
\end{equation}
where $G$ is the finite reflection group generated by $\mathcal{R}_+$ and $\rho(E):=\{\rho {\bf x}:{\bf x}\in E\}$.
Note that if $f\in C_c(\mathbb R^d)$, then by the definition of $\mathcal{D}_j$ we have
$${\rm supp}\left(\mathcal{D}_jf\right)\subset \bigcup_{\rho\in G}\rho\left({\rm supp}(f)\right),\ j=1,\dots, d,$$
Combining with the fact that the set $\bigcup_{\rho\in G}\rho(E)$ is invariant under the action of the
group $G$, it implies that
$${\rm supp}\left(\left(-\Delta_{{\boldsymbol\kappa},0}\right)^vf\right)\subset \bigcup_{\rho\in G}\rho\left({\rm supp}(f)\right),\ v=1,2,\dots.$$
By the definition of the Dunkl-Laplacian-Beltrami operator, for
any $\boldsymbol\xi\in\mathbb S^{d}$, we have
$$\Delta_{{\boldsymbol\kappa},0}\varphi_j({\boldsymbol\xi})=\Delta_{\boldsymbol\kappa}(\varphi_j({\bf x}))\Big|_{{\bf x}={\boldsymbol\xi}},$$
which implies that
\begin{equation}\label{3.35-0}
{\rm supp}\left(\Delta_{{\boldsymbol\kappa},0}\varphi_j\right)\subset \bigcup_{\rho\in G}\rho\Big({\rm c}({\bf x}_j,\frac1N)\Big),\ j=1,\dots, 4n.
\end{equation}
This leads to \eqref{3.35}.

Step 2. By the conclusion of Step 1, we get that for any ${\bf x}\in \mathbb S^d$,
\begin{align}\label{3.36}
  &\quad\#\left\{1\le j\le4n:\,\alpha_j\left(-\Delta_{{\boldsymbol\kappa},0}\right)^v\varphi_j({\bf x})\neq0\right\}\notag
  \\&\le\sum_{j=1}^{4n}\chi_{\bigcup_{\rho\in G}\rho\left({\rm c}\left({\bf x}_j,\frac1N\right)\right)}({\bf x})
  \le\sum_{j=1}^{4n}\sum_{\rho\in G}\chi_{\rho\left({\rm c}\left({\bf x}_j,\frac1N\right)\right)}({\bf x})\notag
  \\&=\sum_{\rho\in G}\sum_{j=1}^{4n}\chi_{\rho\left({\rm c}\left({\bf x}_j,\frac1N\right)\right)}({\bf x})\le \sum_{\rho\in G}1=\#G,
\end{align} where in the last inequality we used the pairwise disjoint property of
 $\left\{\rho\left({\rm c}\left({\bf x}_j,1/N\right)\right)\right\}_{j=1}^{4n}$ fro any $\rho\in G$.

Step 3. It is easy to verify that for any $1\le j\le 4n$ and $v=1,2,\dots$,
$$\left\|\left(-\Delta_{\boldsymbol\kappa}\right)^v\varphi_j\right\|_\infty\lesssim N^{2v},\,$$
which leads to
\begin{equation}\label{3.37}
\left\|\left(-\Delta_{{\boldsymbol\kappa},0}\right)^v\varphi_j\right\|_\infty\lesssim N^{2v}.
\end{equation}It follows from \eqref{3.35-0} and \eqref{3.37} that for $1\le p\le \infty$,
\begin{equation}\label{3.38}
\left\|\left(-\Delta_{{\boldsymbol\kappa},0}\right)^v\varphi_j\right\|_{p,w_{\boldsymbol\kappa}}\lesssim N^{2v-d/p},\ v=1,2,\dots.
\end{equation}

Step 4.
By \eqref{3.34}, \eqref{3.36}, and \eqref{3.38} we obtain
\begin{align*}
  \left\| \left(-\Delta_{{\boldsymbol\kappa},0}\right)^vf_{\boldsymbol\az}\right\|_{p,w_{\boldsymbol\kappa}}
  &\le\left(\int_{\mathbb S^d}\Big|\sum_{j=1}^{4n}\alpha_j\left(-\Delta_{{\boldsymbol\kappa},0}\right)^v\varphi_j({\bf x})\Big|^p{\rm d}\sigma({\bf x})\right)^{1/p}
  \\&\le (\#G)^{1-1/p}\left(\int_{\mathbb S^d}\sum_{j=1}^{4n}|\alpha_j|^p\left|\left(-\Delta_{{\boldsymbol\kappa},0}\right)^v\varphi_j({\bf x})\right|^p{\rm d}\sigma({\bf x})\right)^{1/p}
  \\&\lesssim N^{2v-d/p}\|{\boldsymbol\az}\|_{\ell_p^{4n}}.
\end{align*}This proves \eqref{3.33}.

Now we turn to prove \eqref{3.32}. By the Kolmogorov-type
inequality (see \cite[Theorem 8.1]{Di}), \eqref{3.33} and
\eqref{3.29} that for $v>r$ with $r\in \mathbb N$,
\begin{align*}
  \|\left(-\Delta_{{\boldsymbol\kappa},0}\right)^{r/2}f_{\boldsymbol\alpha}\|_{p,w_{\boldsymbol\kappa}}
  \le \left\|\left(-\Delta_{{\boldsymbol\kappa},0}\right)^{v}f_{\boldsymbol\alpha}\right\|_{p,w_{\boldsymbol\kappa}}^{\frac{r}{2v}}\left\|f_{\boldsymbol\alpha}\right\|_{p,w_{\boldsymbol\kappa}}^{\frac{2v-r}{2v}}
  \lesssim N^{r-d/p}\|{\boldsymbol\az}\|_{\ell_p^{4n}}.
\end{align*}
Hence, the proof of (i) is completed.

(ii) Since $B_\gamma^\Theta(L_{p,w_{\boldsymbol\kappa}})$ is continuously embedded into
$B_\infty^\Theta(L_{p,w_{\boldsymbol\kappa}})$ for $0<\gamma<\infty$, it
only needs to show \eqref{3.31} for $B_\gamma^\Theta(L_{p,w_{\boldsymbol\kappa}})$,
$0<\gamma<\infty$. Notice that $E_{2^j}(f_{\boldsymbol\az})_{p,w_{\boldsymbol\kappa}}\le
\|f_{\boldsymbol\az}\|_{p,w_{\boldsymbol\kappa}}$ for $j\ge0$. Thus,
 we have
\begin{align}\label{3.43}
  \sum\limits_{2^j<N}\Big(\frac{ E_{2^j}(f_{\boldsymbol\az})_{p,w_{\boldsymbol\kappa}}}{\Theta(2^{-j})}\Big)^\gamma
  &\le\sum\limits_{2^j<N}\Big(\frac1{\Theta(2^{-j})}\Big)^\gamma\|f_{\boldsymbol\az}\|_{p,w_{\boldsymbol\kappa}}^\gamma.
\end{align}Similar to \eqref{3.5-000}, we get $$\sum\limits_{2^j<N}\Big(\frac1{\Theta(2^{-j})}\Big)^\gamma\lesssim \Theta(N^{-1})^{-\gamma}.$$
Together with \eqref{3.29} we have
\begin{align}\label{3.43}
  \sum\limits_{2^j<N}\Big(\frac{ E_{2^j}(f_{\boldsymbol\az})_{p,w_{\boldsymbol\kappa}}}{\Theta(2^{-j})}\Big)^\gamma
  &\lesssim \Theta(N^{-1})^{-\gamma}\|f_{\boldsymbol\az}\|_{p,w_{\boldsymbol\kappa}}^\gamma\notag
  \\&\asymp \Theta(N^{-1})^{-\gamma}N^{-d\gamma/p}\|{\boldsymbol\az}\|_{\ell_p^{4n}}^\gamma.
\end{align}
Choose a positive integer $\upsilon>s$.  By the fact that
$\|f\|_{p,w_{\boldsymbol\kappa}}\le \|f\|_p$ and \eqref{2.2} we obtain
 $$E_{2^j}(f_{\boldsymbol\az})_{p,w_{\boldsymbol\kappa}}\le E_{2^j}(f_{\boldsymbol\az})_{p}\lesssim 2^{-j\upsilon}\|f_{\boldsymbol\az}\|_{W_p^\upsilon},\  j\ge0,$$ which combining with \eqref{3.30} for $W_p^\upsilon$, we get
 \begin{align}\label{3.44}
  \sum\limits_{2^j\ge N}\Big(\frac{ E_{2^j}(f_{\boldsymbol\az})_{p,w_{\boldsymbol\kappa}}}{\Theta(2^{-j})}\Big)^\gamma&
  \le\sum\limits_{2^j\ge N}\Big(\frac{2^{-j\upsilon}}{\Theta(2^{-j})}\Big)^\gamma\|f_{\boldsymbol\az}\|_{W_p^\upsilon}^\gamma\nonumber
  \\&\le \Theta(N^{-1})^{-\gamma}N^{-s\gamma}\sum\limits_{2^j\ge N}2^{j(s-\upsilon)\gamma}\|f_{\boldsymbol\az}\|_{W_p^\upsilon}^\gamma\nonumber
  \\&\asymp\Theta(N^{-1})^{-\gamma}N^{-\upsilon\gamma}\|f_{\boldsymbol\az}\|_{W_p^\upsilon}^\gamma\notag
  \\&\lesssim \Theta(N^{-1})^{-\gamma}N^{-d\gamma/p}\|{\boldsymbol\az}\|_{\ell_p^{4n}}^\gamma,
\end{align}where in the second inequality we used
$$\frac{\Theta(N^{-1})}{N^{-s}}\lesssim \frac{\Theta(2^{-j})}{2^{-js}},\ 2^j\ge N,$$since ${\Theta(t)}/{t^s}$ is almost decreasing.
It follows from \eqref{3.29}, \eqref{3.43}, and \eqref{3.44} that
 \begin{align*}
 \|f_{\boldsymbol\az}\|_{B_\gamma^\Theta(L_{p,w_{\boldsymbol\kappa}})}&:=\|f_{\boldsymbol\az}\|_{p,w_{\boldsymbol\kappa}}+\Big( \sum\limits_{j=0}^\infty\Big(\frac{ E_{2^j}(f_{\boldsymbol\az})_{p,w_{\boldsymbol\kappa}}}{\Theta(2^{-j})}\Big)^\gamma\Big)^{1/\gamma}
 \lesssim \Theta(N^{-1})^{-1}N^{-d/p}\|{\boldsymbol\az}\|_{\ell_p^{4n}}.
\end{align*}

This completes the  proof of Lemma \ref{lem44}.
\end{proof}

\begin{proof}[Proof of lower estimates of \eqref{3.23}.]\

The proof will be divided into two cases.

{\it Case 1.} $2\le p\le \infty$. In this case, by the fact that $W_{2,w_{\boldsymbol\kappa}}^r$ is continuously embedded into $W_{p,w_{\boldsymbol\kappa}}^r$, it suffices to consider the case $p=\infty$.
Hence, what's left to show is
\begin{equation}\label{5.5}e_n^{\rm ran}\left({\rm Int}_{\mathbb S^{d},w_{\boldsymbol\kappa}};BW_{\infty,w_{\boldsymbol\kappa}}^r\right)\gtrsim n^{-r/d-1/2}.\end{equation}
Indeed, if $p=\infty$, then it follows from \eqref{3.30} that  for
$\az_j\in\{-1,1\}, \, j=1,\dots, 4n,$
$$\|f_{\boldsymbol\az}\|_{W_{\infty,w_{\boldsymbol\kappa}}^r}\lesssim n^{r/d}\,\|{\boldsymbol\az}\|_{\ell_\infty^{4n}}\lesssim n^{r/d}.$$
Hence, there exists a positive constant $C_1$  independent of $n$
such that $C_1n^{-r/d}f_{\boldsymbol\az}\in
BW_{\infty,w_{\boldsymbol\kappa}}^r$. Set
$$\psi_j({\bf x}):=C_1n^{-r/d}\varphi_j({\bf x}),\ j=1,\dots,4n.$$
Clearly, we have
$$F_1:=\Big\{\sum\limits_{j=1}^{4n}
\az_j\psi_j:\,\alpha_j\in\{-1,1\},\ j=1,\dots,4n\Big\}\subset
BW_{\infty,w_{\boldsymbol\kappa}}^r,$$ and  by \eqref{3.27},
$${\rm Int}_{\mathbb S^d,w_{\boldsymbol\kappa}}(\psi_j)=\int_{\mathbb S^d}\psi_j({\bf x})w_{\boldsymbol\kappa}({\bf x}){\rm d}\sigma({\bf x})=C_1 n^{-r/d}\,\|\varphi_j\|_{1,w_{\boldsymbol\kappa}}\asymp n^{r/d-1}.$$
Applying Lemma \ref{lem3.5} (a) we obtain \eqref{5.5}.

{\it Case 2.} $1\le p<2$. It follows from \eqref{3.29} that
$$\|\pm \varphi_j\|_{W_{p,w_{\boldsymbol\kappa}}^r}\lesssim
n^{r/d-1/p}.$$Hence, there exists a positive constant $C_2$
independent of $n$ such that
$$\psi_j({\bf x}):=C_2n^{-r/d+1/p}\varphi_j({\bf x})\in BW_{p,w_{\boldsymbol\kappa}}^r,\ j=1,\dots,4n.$$
Clearly, we have$$F_2:=\{\pm\psi_j:\ j=1,\dots,4n\}\subset
BW_{p,w_{\boldsymbol\kappa}}^r,$$and  by \eqref{3.27},
$${\rm Int}_{\mathbb S^d,w_{\boldsymbol\kappa}}(\psi_j)=\int_{\mathbb S^d}\psi_j({\bf x})w_{\boldsymbol\kappa}({\bf x}){\rm d}\sigma({\bf x})=C_2 n^{-r/d+1/p}\|\varphi_j\|_{1,w_{\boldsymbol\kappa}}\asymp n^{-r/d+1/p-1}.$$
Applying Lemma \ref{lem3.5} (b), we obtain for $1\le p<2$,
\begin{equation*}e_n^{\rm ran}\left({\rm Int}_{\mathbb S^{d},w_{\boldsymbol\kappa}};BW_{p,w_{\boldsymbol\kappa}}^r\right)\gtrsim n^{-r/d+1/p-1},\end{equation*} which, together with
\eqref{5.5}, gives  the lower bounds of $e_{n}^{\rm
ran}(BW_{p,w_{\boldsymbol\kappa}}^r)$ for $1\le p\le\infty$.

The proof of lower estimates of \eqref{3.23} is completed.
\end{proof}

\begin{proof}[Proof of lower estimates of \eqref{3.25}.] \

The proof will be divided into two cases.

{\it Case 1.} $2\le p\le \infty$. In this case, we want to prove
$$e_n^{\rm ran}({\rm Int}_{\mathbb S^{d},w_{\boldsymbol\kappa}};BB_\gamma^\Theta(L_{p,w_{\boldsymbol\kappa}}))\gtrsim \Theta\left(n^{-1/d}\right)n^{-1/2}.$$
It suffices to consider the case $p=\infty$ since the embedding relation $$B_\gamma^\Theta(L_\infty)\hookrightarrow B_\gamma^\Theta(L_{p,w_{\boldsymbol\kappa}}),\quad2\le p<\infty.$$
Hence, what's left to show is
\begin{equation}\label{3.47}e_n^{\rm ran}({\rm Int}_{\mathbb S^{d},w_{\boldsymbol\kappa}};BB_\gamma^\Theta(L_\infty))\gtrsim \Theta\left(n^{-1/d}\right)n^{-1/2}.\end{equation}
Indeed, if $p=\infty$, then
it follows from \eqref{3.31} that  when $\az_j\in\{-1,1\}, \,
j=1,\dots, 4n,$
$$\|f_{\boldsymbol\az}\|_{B_\gamma^\Theta(L_\infty)}\lesssim \Theta\left(n^{-1/d}\right)^{-1}\,\|{\boldsymbol\az}\|_{\ell_\infty^{4n}}\lesssim \Theta\left(n^{-1/d}\right)^{-1}.$$
Hence, there exists a positive constant $C_1$ independent of $n$
such that
$C_1\Theta\left(n^{-1/d}\right)^{-1}f_{\boldsymbol\az}\in
BB_\gamma^\Theta(L_\infty)$. Set
$$\psi_j({\bf x}):=C_1\Theta\left(n^{-1/d}\right)\varphi_j({\bf x}),\ j=1,\dots,4n.$$
Then we have
$$F_1:=\Big\{\sum\limits_{j=1}^{4n}
\az_j\psi_j:\,\alpha_j\in\{-1,1\},\
j=1,\dots,4n\Big\}\subset BB_\gamma^\Theta(L_\infty).$$Meanwhile, it follows from
\eqref{3.27} that
$${\rm Int}_{\mathbb S^d,w_{\boldsymbol\kappa}}(\psi_j)=\int_{\mathbb S^d}\psi_j({\bf x})w_{\boldsymbol\kappa}({\bf x}){\rm d}\sigma({\bf x})=C_1 \Theta\left(n^{-1/d}\right)\,\|\varphi_j\|_{1,w_{\boldsymbol\kappa}}\asymp \Theta\left(n^{-1/d}\right)n^{-1}.$$
Applying Lemma \ref{lem3.5} (a) we obtain \eqref{3.47}.

{\it Case 2.} $1\le p<2$. In this case, we want to prove
\begin{equation}\label{5.5-5}
e_n^{\rm ran}({\rm Int}_{\mathbb S^{d},w_{\boldsymbol\kappa}};BB_\gamma^\Theta(L_{p,w_{\boldsymbol\kappa}}))\gtrsim \Theta\left(n^{-1/d}\right)n^{1/p-1}.
\end{equation}
Indeed, it follows from \eqref{3.31} that $$\|\pm
\varphi_j\|_{B_\gamma^\Theta(L_{p,w_{\boldsymbol\kappa}})}\lesssim
\Theta\left(n^{-1/d}\right)n^{-1/p}.$$Hence, there exists a
positive constant $C_2$ independent of $n$ such that
$$\psi_j({\bf x}):=C_2\Theta\left(n^{-1/d}\right)n^{1/p}\varphi_j({\bf x})\in BB_\gamma^\Theta(L_{p,w_{\boldsymbol\kappa}}),\ j=1,\dots,4n.$$
Then we have$$F_2:=\{\pm\psi_j:\ j=1,\dots,4n\}\subset
BB_\gamma^\Theta(L_{p,w_{\boldsymbol\kappa}}),$$and by
\eqref{3.27},
$${\rm Int}_{\mathbb S^d,w_{\boldsymbol\kappa}}(\psi_j)=\int_{\mathbb S^d}\psi_j({\bf x})w_{\boldsymbol\kappa}({\bf x}){\rm d}\sigma({\bf x})=C_2 \Theta\left(n^{-1/d}\right)n^{1/p}\|\varphi_j\|_{1,w_{\boldsymbol\kappa}}\asymp \Theta\left(n^{-1/d}\right)n^{1/p-1}.$$
Applying Lemma \ref{lem3.5} (b) we obtain \eqref{5.5-5}.

This completes the proof.
\end{proof}

\section{Extensions to the unit ball and simplex}\label{sect4}

In this section, we will extend the results of the previous
sections to the unit ball and standard simplex. For this purpose,
we shall describe briefly some necessary notations and results on
these domains. Unless otherwise stated, most of the results
described in this section can be found in the paper \cite{DW} and
the book \cite{DX}.
\subsection{Weights functions on \texorpdfstring{$\Omega^d$}.}

Let $\Oz^d$ be either the unit ball $$\mathbb B^{d}:=\{{\bf x}\in \mathbb{R}^{d}:\,\|{\bf x}\|^2=x_1^2+\dots+x_{d}^2\le1\}$$ or the simplex
$$\mathbb T^{d}:=\{{\bf x}\in \mathbb{R}^{d}:\,|{\bf x}|:=x_1+\dots+x_d\le1,\ x_1,\dots,x_d\ge0\}$$ and let ${\rm d}{\bf x}$ be the usual Lebesgue measure on $\Oz^d$ normalized by $\int_{\Oz^d}{\rm d}{\bf x}=1$.
For simplicity, we denote ${\rm meas}(E):=\int_{ E}{\rm d}{\bf x}$ of a subset $E$ of $\Oz^d$.
The \emph{distance }of two points ${\bf x}=(x_1,\dots,x_d),\,{\bf y}=(y_1,\dots,y_d)\in \Oz^d$ is defined by
$$d_{\Oz}({\bf x},{\bf y}):=\Bigg\{
\begin{aligned}&\arccos\Big(\langle{\bf x}, {\bf y}\rangle+\sqrt{1-\|{\bf x}\|^2}\sqrt{1-\|{\bf y}\|^2}\Big),&&{\rm if}\ \Oz^d=\mathbb B^d,
\\
&\arccos\Big(\sum_{j=1}^d\sqrt{x_j\,y_j}+\sqrt{1-|{\bf x}|}\sqrt{1-|{\bf y}|}\Big),&&{\rm if}\ \Oz^d=\mathbb T^d.\end{aligned}$$
And we denote by $C_{\Oz}({\bf x},\tz)$ the ball centred at ${\bf x}\in \Oz^d$ and of radius $\tz>0$ in the metric space $(\Oz^d,d_{\Oz})$, that is
$$
C_{\Oz}({\bf x},\tz):=\{{\bf y}\in \Oz^d:\ d_{\Oz}({\bf x},{\bf y})\le\tz\}.
$$

A weight function $w$ on $\Oz^d$ is called \emph{a doubling weight} if
$$\sup\limits_{{\bf x}\in \Oz^d,\,\tz>0}\frac{w(C_{\Oz}({\bf x},2\tz))}{w(C_{\Oz}({\bf x},\tz))}=:L_w<\infty,$$where $w({ E})
:=\int_{ E}w({\bf y}){\rm d}{\bf y}$ for a measurable subset ${
E}$ of $\Oz^d$; and it is call an \emph{$A_\infty$ weight} if
there exists $\beta\ge1$ such that
$$\frac{w(C_{\Oz}({\bf x},\tz))}{w({ E})}\le\beta\left(\frac{{\rm meas}(C_{\Oz}({\bf x},2\tz))}{{\rm meas}({ E})}\right)^\beta,$$
for every measurable subset ${ E}$ of an arbitrarily given ball
$C_{\Oz}({\bf x},\tz)$ in $\Oz^d$.

Integrations on $\Oz^d$ are closely related to integrations on the sphere $\mathbb S^d$. Given a real function $f$ on $\Oz^d$, for any $\bar{{\bf x}}=({\bf x},x_{d+1})=(x_1,\dots,x_{d+1})\in \mathbb S^d$, we define
\begin{equation}\label{4.1}
T_{\Oz}f(\bar{\bf x}):=\Bigg\{
\begin{aligned}&\frac12f({\bf x})|x_{d+1}|,&&{\rm if}\ \Oz^d=\mathbb B^d,
\\
&\frac12f({\bf x}^2)\prod_{j=1}^{d+1}x_j,&&{\rm if}\ \Oz^d=\mathbb T^d,\end{aligned}
\end{equation}
and
\begin{equation}\label{4.1-0}
\psi(\bar{{\bf x}}):=\Big\{
\begin{aligned}&{\bf x},&&{\rm if}\ \Oz^d=\mathbb{B}^d,
\\
&{\bf x}^2,&&{\rm if}\ \Oz^d=\mathbb{T}^d,\end{aligned}
\end{equation}
where ${\bf x}^2=(x_1^2,\dots,x_d^2)$. Then we obtain
\begin{equation}\label{4.2}
\int_{\Oz^d}f({\bf x})w({\bf x}){\rm d}{\bf x}=\int_{\mathbb{S}^d}f\circ \psi(\bar{{\bf x}})\,T_{\Oz}w(\bar{{\bf x}}){\rm d}\sigma(\bar{{\bf x}}),
\end{equation}where $f\circ \psi(\bar{{\bf x}}):=f(\psi(\bar{{\bf x}}))$ for $\bar{\bf x}\in\mathbb S^d$.
By the relation \eqref{4.2} it has been proved in \cite[Theorem 2]{DW} that a weight function $w$ on $\Oz^d$ is a doubling weight (resp. an $A_\infty$ weight)
if and only if $T_{\Oz}w$ is a doubling weight (resp. an $A_\infty$ weight) on $\mathbb S^d$. Furthermore, if $w$ is a doubling weight
then $s_w>0$ satisfies
$$
\sup\limits_{{\bf y}\in \Oz^d,\,\tz>0}\frac{w(C_{\Oz}({\bf y},2^m\tz))}{w(C_{\Oz}({\bf y},\tz))}\le C_{L_w}2^{ms_w},\ m=1,2,\dots,
$$if and only if it satisfies
$$
\sup\limits_{{\bf x}\in\mathbb S^d,\,r>0}\frac{T_{\Oz}w({\rm c}({\bf x},2^mr))}{T_{\Oz}w({\rm c}({\bf x},r))}\le C_{L_w}'2^{ms_w},\ m=1,2,\dots.
$$

Given $\boldsymbol\kappa=(\kappa_1,\dots,\kappa_{m})\in[0,+\infty)^{m}$, $\mu\ge0$, and ${\bf v}_j\in\mathbb S^{d-1}$, $j=1,\dots,m$, let $h_{\boldsymbol\kappa}^2({\bf x})$ be given by \eqref{2.0} and be even in each of its variables if $\Oz^d=\mathbb T^d$. Consider the following \emph{product weights} on $\Oz^d$:
\begin{equation}\label{4.2-0}
w_{\boldsymbol\kappa,\mu}^{\Oz}({\bf x}):=\Bigg\{
\begin{aligned}&(1-\|{\bf x}\|^2)^{\mu-\frac12}h_{\boldsymbol\kappa}^2({\bf x}),&&{\rm if}\ \Oz^d=\mathbb B^d,
\\
&(1-|{\bf x}|)^{\mu-\frac12}h_{\boldsymbol\kappa}^2(\sqrt{\bf x})\prod_{j=1}^dx_j^{-\frac12},&&{\rm if}\ \Oz^d=\mathbb T^d,\end{aligned}
\end{equation}
for ${\bf x}=(x_1,\dots,x_d)\in\Omega^d$, where $\sqrt{{\bf x}}=(\sqrt{x_1},\dots,\sqrt{x_d})$ for ${\bf x}\in\mathbb T^d$. Clearly,
\begin{equation}\label{4.2-00}
T_{\Oz}w_{\boldsymbol\kappa,\mu}^{\Oz}(\bar{{\bf
x}})=|x_{d+1}|^{2\mu}\prod_{j=1}^{m}\left|\langle\bar{{\bf
x}},\bar{{\bf v}}_j\rangle\right|^{2\kappa_j}
=\prod_{j=1}^{m+1}\left|\langle\bar{{\bf x}},\bar{{\bf
v}}_j\rangle\right|^{2\kappa_j}=:h_{\boldsymbol\kappa,\mu}^2(\bar{\bf
x})=h_{\bar {\boldsymbol\kappa}}^2(\bar{\bf x}),
\end{equation}
for $\bar{{\bf x}}=({\bf x},x_{d+1})\in\mathbb S^d,$ where $\bar{{\bf v}}_{j}=({\bf v}_{j},0)\in\mathbb S^d$, $j=1,\dots,m$, $\bar{{\bf v}}_{m+1}={\bf e}_{d+1}$, and $\bar{\boldsymbol\kappa}=(\boldsymbol\kappa,\mu)$.
Thus, we conclude that $w_{\boldsymbol\kappa,\mu}^{\Oz}$ is an $A_\infty$ weight on $\Oz^d$.
Specially, if $\mathcal{R}:=\{\pm\bar {\bf v}_j\}_{j=1}^m$ forms a root system on $\mathbb R^{d+1}$ and $h_{\boldsymbol\kappa,\mu}^2$ is invariant under the reflection group $G$ generated by $\mathcal{R}$, then \eqref{4.2-0} is called \emph{Dunkl weight} on $\Omega^d$.
As an example, when $m=d$ and $\bar{{\bf v}}_{j}={\bf e}_j,j=1,\dots,d$, i.e.,
$$w_{\boldsymbol\kappa,\mu}^{\Oz}({\bf x}):=\Bigg\{
\begin{aligned}&(1-\|{\bf x}\|^2)^{\mu-\frac12}\prod_{j=1}^d|x_j|^{2\kappa_j},&&{\rm if}\ \Oz^d=\mathbb B^d,
\\
&(1-|{\bf x}|)^{\mu-\frac12}\prod_{j=1}^{d}x_j^{\kappa_j-\frac12},&&{\rm if}\ \Oz^d=\mathbb T^d,\end{aligned}$$by \eqref{2.000} and \eqref{4.2-00} we have
$$s_{w_{\boldsymbol\kappa,\mu}^{\Oz}}:=d+2|\boldsymbol\kappa|+2\mu-2\min\{\mu,\kappa_j,j=1,\dots,d\}.$$

\subsection{Function Spaces on \texorpdfstring{$\Omega^d$}.}

For $1\le p<\infty$, denote by $L_{p,w}(\Oz^d)$ the  weighted
Lebesgue space on $\Oz^d$ equipped with finite quasi-norm
$$\|f\|_{L_{p,w}(\Oz^d)}:=\left(\int_{\Oz^d}|f({\bf x})|^pw({\bf x}){\rm d}{\bf x}\right)^{1/p},$$
while in the case of $p=\infty$ we consider the space $C(\Oz^d)$
of continuous functions on $\Oz^d$ with uniform norm
$\|f\|_{L_{\infty}(\Oz^d)}:=\|f\|_{L_{\infty,\oz}(\Oz^d)}$.
 We denote by $\Pi_n(\Oz^d)$ the space of the restrictions to $\Oz^d$ of all real algebra polynomials in $d$ variables of degree no more than $n$.
 Note that $\dim \Pi_n(\Oz^d)\asymp n^d$.
 The best approximation of $f\in L_{p,w}(\Oz^d)$ from $\Pi_n(\Oz^d)$ in  $L_{p,w}(\Omega^d)$ by
$$E_n(f)_{L_{p,w}(\Oz^d)}:=\inf\limits_{g\in\Pi_n(\Oz^d)}\|f-g\|_{L_{p,w}(\Oz^d)}.$$

\subsubsection{Weighted Besov classes on \texorpdfstring{$\Omega^d$}.}

Let $\Theta\in\Phi_s^*$. For $1\le p\le\infty$ and $0<\gamma\le\infty$, we say $f\in B_{\gamma}^\Theta(L_{p,w}(\Oz^d))$ if $f\in L_{p,w}(\Omega^d)$  and
$$\|f\|_{B_{\gamma}^\Theta(L_{p,w}(\Oz^d))}:=\|f\|_{L_{p,w}(\Omega^d)}+\left\{\sum_{j=0}^\infty\left(\frac{E_{2^j}(f)_{L_{p,w}(\Omega^d)}}{\Theta(2^{-j})}\right)^\gamma\right\}^{1/\gamma}$$
is finite. Similar to the proof of Proposition \ref{thm2.1}, we can prove that if $\Theta_1(t)$, $\Theta(t)=t^r\Theta_1(t)\in\Phi_s^*$, and $r>(1/p-1/q)_+s_w$, then $B_\gamma^\Theta(L_{p,w}(\Oz^d))\hookrightarrow L_{q,w}(\Oz^d)$.

Using the relation \eqref{4.2} we have
\[
  E_n(f)_{L_{p,w}(\Oz^d)}=\Bigg\{
\begin{aligned}&E_{2n}(f\circ \psi)_{L_{p,T_\Oz w}(\mathbb S^d)},&&{\rm if}\ \Oz^d=\mathbb{B}^d,
\\
&E_n(f\circ \psi)_{L_{p,T_\Oz w}(\mathbb S^d)},&&{\rm if}\ \Oz^d=\mathbb{T}^d,\end{aligned}
\]which leads that for all $B_\gamma^\Theta(L_{p,w}(\Oz^d))$,
\begin{equation}\label{4.11}
\|f\|_{B_\gamma^\Theta(L_{p,w}(\Oz^d))}\asymp \|f\circ \psi\|_{B_\gamma^\Theta(L_{p,T_\Omega w}(\mathbb{S}^d))}.
\end{equation}

\subsubsection{Weighted Sobolev classes on \texorpdfstring{$\Omega^d$}.}

Consider the Dunkl weight function
$$w_{\boldsymbol\kappa,\mu}^{\mathbb B}({\bf
x}):=h_{\boldsymbol\kappa}^2({\bf x})\left(1-\|{\bf
x}\|^2\right)^{\mu-1/2},\ \mu\ge0,\ {\bf x}\in\mathbb{B}^d,$$
where $h_{\boldsymbol\kappa}^2$ is given by \eqref{2.0} and is
reflection invariant under the finite reflection group $G$. For
simplicity, we denote $L_2(w_{\boldsymbol\kappa,\mu}^{\mathbb
B})\equiv L_{2,w_{\boldsymbol\kappa,\mu}^{\mathbb B}}(\mathbb
B^d)$. Let $\mathcal{V}_n^d(w_{\boldsymbol\kappa,\mu}^{\mathbb
B})$ denote the space of weighted orthogonal polynomials of degree
$n$ with respect to the inner product
$$
\langle f,g\rangle_{L_2(w_{\boldsymbol\kappa,\mu}^{\mathbb
B})}:=\int_{\mathbb B^d}f({\bf x})g({\bf x})
w_{\boldsymbol\kappa,\mu}^{\mathbb B}({\bf x}){\rm d}{\bf x},\
f,g\in L_2(w_{\boldsymbol\kappa,\mu}^{\mathbb B}),
$$i.e., $$\mathcal{V}_n^d(w_{\boldsymbol\kappa,\mu}^{\mathbb
B})=\{f\in \Pi_n(\mathbb B^d) :\ \langle
f,g\rangle_{L_2(w_{\boldsymbol\kappa,\mu}^{\mathbb B})}=0, \ {\rm
for\ all}\ g\in \Pi_{n-1}(\mathbb B^d)\}.
$$ Elements of $\mathcal{V}_n^d(w_{\boldsymbol\kappa,\mu}^{\mathbb
B})$ are closely related to the $h$-spherical harmonics associated
with the Dunkl weight function
$$
 h^2_{\boldsymbol\kappa,\mu}({\bf x},x_{d+1}):=h^2_{\boldsymbol\kappa}({\bf x})|x_{d+1}|^{2\mu},\ ({\bf x},x_{d+1})\in\mathbb S^{d},
$$ which is invariant under the group $G\times \mathbb{Z}_2$.

Let $Y_n$ be such an $h$-harmonic
polynomial of degree $n$ and assume that $Y_n$ is even in the $(d+1)$-th variable; that is, $$Y_n({\bf x},x_{d+1})=Y_n({\bf x},-x_{d+1}).$$
Then we can write
$$Y_n({\bf y})=r^nP_n({\bf x}),\ {\bf y}=r({\bf x},x_{d+1})\in\mathbb R^{d+1},\ r=\|{\bf y}\|,\ ({\bf x},x_{d+1})\in\mathbb S^{d},$$
in polar coordinates. Then $P_n\in\mathcal{V}_n^d(w_{\boldsymbol\kappa,\mu}^{\mathbb B})$ and this relation is
an one-to-one correspondence.
Furthermore, let $\Delta_{\boldsymbol\kappa,\mu}$ denote the Dunkl-Laplacian
associated with $h^2_{\boldsymbol\kappa,\mu}$ and $\Delta_{\boldsymbol\kappa,\mu,0}$ denote the corresponding spherical Dunkl-Laplacian.
When $\Delta_{\boldsymbol\kappa,\mu,0}$
is applied to functions on $\mathbb R^{d+1}$ that are even in the $(d+1)$-th variable,
the spherical Dunkl-Laplacian can be written in polar coordinates ${\bf y}=r({\bf x},x_{d+1})$ as
$$\Delta_{\boldsymbol\kappa,\mu,0}=\Delta_{\boldsymbol\kappa}-\langle{\bf x},\nabla\rangle^2-2\lambda_{\boldsymbol\kappa,\mu}\langle{\bf x},\nabla\rangle,\ \ \lambda_{\boldsymbol\kappa,\mu}=|\boldsymbol\kappa|+\mu+\frac{d-1}2,$$
in which the operators $\Delta_{\boldsymbol\kappa}$ and $\nabla=(\partial_1,\dots,\partial_d)$ are all acting on ${\bf x}$ variables, and
$\Delta_{\boldsymbol\kappa}$ is the Dunkl-Laplacian associated with $h_{\boldsymbol\kappa}^2$ on $\mathbb R^d$. Define
$$D_{\boldsymbol\kappa,\mu}^{\mathbb B}:=\Delta_{\boldsymbol\kappa}-\langle{\bf x},\nabla\rangle^2-2\lambda_{\boldsymbol\kappa,\mu}\langle{\bf x},\nabla\rangle,$$
which satisfies
$$D_{\boldsymbol\kappa,\mu}^{\mathbb B}P=-n(n+2\lambda_{\boldsymbol\kappa,\mu})P,\ \ P\in\mathcal{V}_n^d(w_{\boldsymbol\kappa,\mu}^{\mathbb B}).$$

For $f \in L_p(w_{\boldsymbol\kappa,\mu}^{\mathbb B})$, define $F({\bf x},x_{d+1}) = f({\bf x})$. Then
$F \in L_p(w_{\boldsymbol\kappa,\mu}^{\mathbb S})$ and for $r>0$,
\begin{equation} \label{4.3}
\|(-D_{\boldsymbol\kappa,\mu}^{\mathbb B})^{r/2} f\|_{L_p(w_{\boldsymbol\kappa,\mu}^{\mathbb B})} =
      \|(-\Delta_{\boldsymbol\kappa,\mu,0})^{r/2}F\|_{L_p(w_{\boldsymbol\kappa,\mu}^{\mathbb S})},
\end{equation}which, combining with \eqref{2.1}, leads that Bernstein's inequality holds: for $P\in\Pi_n^d$,
\begin{equation} \label{4.4}
\|(-D_{\boldsymbol\kappa,\mu}^{\mathbb B})^{r/2} P\|_{L_p(w_{\boldsymbol\kappa,\mu}^{\mathbb B})}\lesssim
      n^r\,\|P\|_{L_p(w_{\boldsymbol\kappa,\mu}^{\mathbb B})}.
\end{equation}
Therefore, for $1\le p\le \infty$ and $r>0$, we can define the weighted
Sobolev space $W_{p}^r(w_{\boldsymbol\kappa,\mu}^{\mathbb B})$ to be the space of all real
functions $f$ with finite norm
$$\|f\|_{W_{p}^r(w_{\boldsymbol\kappa,\mu}^{\mathbb B})}:=\|f\|_{L_p(w_{\boldsymbol\kappa,\mu}^{\mathbb B})}+\|(-D_{\boldsymbol\kappa,\mu}^{\mathbb B})^{r/2}f\|_{L_{p}(w_{\boldsymbol\kappa,\mu}^{\mathbb B})},$$
while the weighted Sobolev class
$BW_{p}^r(w_{\boldsymbol\kappa,\mu}^{\mathbb B})$ is defined to be the unit ball of the
weighted Sobolev space $W_{p}^r(w_{\boldsymbol\kappa,\mu}^{\mathbb B})$.
 Similar to the spherical case, we have $W_p^r(w_{\boldsymbol\kappa,\mu}^{\mathbb B})\hookrightarrow C(\mathbb B^d)$ if $r>s_{w_{\boldsymbol\kappa,\mu}^{\mathbb B}}/p$, and it follows from \eqref{2.2} and \eqref{4.3} that Jackson's inequality holds: for  $f\in W_p^r(w_{\boldsymbol\kappa,\mu}^{\mathbb B})$,
  \begin{equation}\label{4.5}
  E_n(f)_{L_p(w_{\boldsymbol\kappa,\mu}^{\mathbb B})}\lesssim n^{-r} \|f\|_{W_p^r(w_{\boldsymbol\kappa,\mu}^{\mathbb B})}.
\end{equation}

  Consider the Dunkl weight function
$$
w_{\boldsymbol\kappa,\mu}^{\mathbb T}({\bf x}) = h_{\boldsymbol\kappa}^2(\sqrt{\bf x})
   (1-|{\bf x}|)^{\mu-1/2} /\sqrt{x_1\dots x_d},\ \mu \ge0,\ {\bf x}\in\mathbb T^d,
$$
where  $h_{\boldsymbol\kappa}^2$ given by \eqref{2.0} is
 invariant under the finite reflection group $G$ and is even in
each of its variables. For example, the finite reflection  group
$G$ can be chosen as the rotation group $\mathbb Z_2^d$. For
simplicity, we denote $L_2(w_{\boldsymbol\kappa,\mu}^{\mathbb
T})\equiv L_{2,w_{\boldsymbol\kappa,\mu}^{\mathbb T}}(\mathbb
T^d)$. Let $\mathcal{V}_n^d(w_{\boldsymbol\kappa,\mu}^{\mathbb
T})$ be the space of weighted orthogonal polynomials on $\mathbb
T^d$ of degree $n$ with respect to the inner product
$$
\langle f,g\rangle_{L_2(w_{\boldsymbol\kappa,\mu}^{\mathbb T})}:=\int_{\mathbb B^d}f({\bf x})g({\bf x})w_{\boldsymbol\kappa,\mu}^{\mathbb T}({\bf x}){\rm d}{\bf x},\ f,g\in L_2(w_{\boldsymbol\kappa,\mu}^{\mathbb T}).
$$

 For a polynomial $P_{2n}\in\Pi_{2n}^d$ being even in each of its variables,
we can write $P_{2n}$ as
$$P_{2n} (x_1,\dots, x_d) = R_n(x_1^2,\dots,x_d^2)$$for ${\bf x}=(x_1,\dots, x_d)\in\mathbb T^d$.
This implies that
$P_{2n}\in\mathcal{V}_{2n}^d(w_{\boldsymbol\kappa,\mu}^{\mathbb T})$ if and only if $R_n\in\mathcal{V}_n^d(w_{\boldsymbol\kappa,\mu}^{\mathbb B})$, and the
relation is an one-to-one correspondence. In particular, applying
$D_{\boldsymbol\kappa,\mu}^{\mathbb B}$ on $P_{2n}$ leads to a second order
differential-difference operator, denoted
by $D_{\boldsymbol\kappa,\mu}^{\mathbb T}$, acting on $R_n$. Then
$$
  D_{\boldsymbol\kappa,\mu}^{\mathbb T} P = - n (n+\lambda_{\boldsymbol\kappa,\mu}) P,
    \ P \in\mathcal{V}_n^d(w_{\boldsymbol\kappa,\mu}^{\mathbb T}),\ \lambda_{\boldsymbol\kappa,\mu} =
     |\boldsymbol\kappa| + \mu + \frac{d-1}{2}.
$$

For $f \in L_p(w_{\boldsymbol\kappa,\mu}^{\mathbb T})$, define $F({\bf x}) = f({\bf x}^2)$. Then
$F \in L_p(w_{\boldsymbol\kappa,\mu}^{\mathbb B})$ and for $r>0$,
\begin{equation} \label{4.6}
\|(-D_{\boldsymbol\kappa,\mu}^{\mathbb T})^{r/2} f\|_{L_p(w_{\boldsymbol\kappa,\mu}^{\mathbb T})} =
     2^{-r} \|(-D_{\boldsymbol\kappa,\mu}^{\mathbb B})^{r/2} F\|_{L_p(w_{\boldsymbol\kappa,\mu}^{\mathbb B})},
\end{equation}which, combining with \eqref{2.1}, gives  Bernstein's inequality: for  $P\in\Pi_n^d$,
\begin{equation} \label{4.7}
\|(-D_{\boldsymbol\kappa,\mu}^{\mathbb T})^{r/2} P\|_{L_p(w_{\boldsymbol\kappa,\mu}^{\mathbb T})}\lesssim
      n^r\,\|P\|_{L_p(w_{\boldsymbol\kappa,\mu}^{\mathbb T})}.
\end{equation}
Therefore, for $1\le p\le \infty$ and $r>0$, we can define the weighted
Sobolev space $W_{p}^r(w_{\boldsymbol\kappa,\mu}^{\mathbb T})$ to be the space of all real
functions $f$ with finite norm
$$\|f\|_{W_{p}^r(w_{\boldsymbol\kappa,\mu}^{\mathbb T})}:=\|f\|_{L_p(w_{\boldsymbol\kappa,\mu}^{\mathbb T})}+\|(-D_{\boldsymbol\kappa,\mu}^{\mathbb T})^{r/2}f\|_{L_{p}(w_{\boldsymbol\kappa,\mu}^{\mathbb T})},$$
while the weighted Sobolev class
$BW_{p}^r(w_{\boldsymbol\kappa,\mu}^{\mathbb T})$ is defined to be the unit ball of the
weighted Sobolev space $W_{p}^r(w_{\boldsymbol\kappa,\mu}^{\mathbb T})$.
 Similar to the spherical case, we have $W_p^r(w_{\boldsymbol\kappa,\mu}^{\mathbb T})\hookrightarrow C(\mathbb T^d)$ if $r>s_{w_{\boldsymbol\kappa,\mu}^{\mathbb T}}/p$, and it follows from \eqref{2.2} and \eqref{4.6} that Jackson's inequality holds: for  $f\in W_p^r(w_{\boldsymbol\kappa,\mu}^{\mathbb T})$,
  \begin{equation}\label{4.8}
  E_n(f)_{L_p(w_{\boldsymbol\kappa,\mu}^{\mathbb T})}\lesssim n^{-r} \|f\|_{W_p^r(w_{\boldsymbol\kappa,\mu}^{\mathbb T})}.
\end{equation}

\subsection{Main Results}

\subsubsection{Deterministic case.}

\begin{thm}\label{thm4.1}Let $1\le p\le\infty$, $0<\gamma\le\infty$, $\Theta_1(t)$, $\Theta(t)=t^r\Theta_1(t)\in\Phi_s^*$, and let $w$ be an $A_\infty$ weight with the critical index $s_w$ on $\Oz^d$. Then, for $r>s_w/p$, we have
\begin{equation}\label{4.13}
e_n^{\rm det}\left({\rm Int}_{\Oz^d,w}; BB_{\gamma}^\Theta(L_{p,w}(\Oz^d))\right)\asymp \Theta\big(n^{-\frac1{d}}\big).
\end{equation}Moreover,  the upper estimate of \eqref{4.13} holds for all doubling weights.

Furthermore, if $w_{\boldsymbol\kappa,\mu}^{\Omega}$ is a Dunkl weight with the critical index $s_{w_{\boldsymbol\kappa,\mu}^{\Omega}}$ on $\Oz^d$, then for $r>s_{w_{\boldsymbol\kappa,\mu}^{\Omega}}/p$, we have
\begin{equation}\label{4.15}
e_n^{\rm det}\left({\rm Int}_{\Oz^{d},w_{\boldsymbol\kappa,\mu}^{\Omega}}; BW_p^r(w_{\boldsymbol\kappa,\mu}^{\Omega})\right)\asymp n^{-\frac r{d}}.
\end{equation}
\end{thm}
\begin{proof} The proof of the lower estimates is almost identical to that of  lower estimates
of spherical case, which depends on an analog of Lemma
\ref{lem3.4} on $\Omega^d$ (see \cite[Proposition 4.8]{DW}).
Meanwhile, by the embedding
$W_p^r(w_{\boldsymbol\kappa,\mu}^{\Omega})\hookrightarrow
B_{\infty}^r(L_p(w_{\boldsymbol\kappa,\mu}^{\Omega}))$ we have
$$e_n^{\rm det}\left({\rm Int}_{\Oz^{d},w_{\boldsymbol\kappa,\mu}^{\Omega}}; BW_p^r(w_{\boldsymbol\kappa,\mu}^{\Omega})\right)\lesssim e_n^{\rm det}\left({\rm Int}_{\Oz^d,w_{\boldsymbol\kappa,\mu}^{\Omega}}; BB_\infty^r(L_p(w_{\boldsymbol\kappa,\mu}^{\Omega}))\right),$$
which indicates that the upper estimate of \eqref{4.15} can be
derived from \eqref{1.7}.

Now we turn to give the proof of the upper estimate of \eqref{4.13} for doubling weights.
 To this end, let  $w$ be a doubling weight on $\Omega^d$ and $f\in B_\gamma^\Theta(L_{p,w}(\Oz^d))$. Then, by \eqref{4.11} we have  $F=f\circ \psi\in B_\gamma^\Theta(L_{p,T_\Oz w}(\mathbb S^d))$.
It follows from  \eqref{3.1} that there exist ${\bf x}_1^*,\dots,{\bf x}_n^*\in \mathbb S^d$ and $\lambda_1^*,\dots,\lambda_n^*\in\mathbb R$ such that
$$\left|\int_{\mathbb S^d}F(\bar{{\bf x}})T_{\Oz}w(\bar{{\bf x}}){\rm d}\sigma(\bar{{\bf x}})-\sum_{j=1}^n\lambda_j^* F({\bf x}_j^*)\right|\lesssim \Theta\big(n^{-\frac1{d}}\big).$$
Choose ${\bf y}_j^*=\psi({\bf x}_j^*)\in\Omega^d,j=1,\dots,n$. Then by \eqref{4.2} we have
$$\left|\int_{\Oz^d}f({\bf x})w({\bf x}){\rm d}{\bf x}-\sum_{j=1}^n\lambda_j^* f({\bf y}_j^*)\right|\lesssim \Theta\big(n^{-\frac1{d}}\big),$$
which yields the desired estimate.
This completes the proof.
\end{proof}

\subsubsection{Randomized case.}

\begin{thm}\label{thm4.2} Let $1\le p\le\infty$, $0<\gamma\le\infty$, $\Theta_1(t)$, $\Theta(t)=t^r\Theta_1(t)\in\Phi_s^*$,
and let $w_{\boldsymbol\kappa,\mu}^{\Omega}$ be a product weight
with the critical index $s_{w_{\boldsymbol\kappa,\mu}^{\Omega}}$
on $\Oz^d$. Then, for
$r>s_{w_{\boldsymbol\kappa,\mu}^{\Omega}}/p$, we have
\begin{equation}\label{4.16}
e_n^{\rm ran}\left({\rm Int}_{\Oz^d,w_{\boldsymbol\kappa,\mu}^{\Omega}}; BB_{\gamma}^\Theta(L_{p,w_{\boldsymbol\kappa,\mu}^{\Omega}}(\Oz^d))\right)\asymp \Theta\big(n^{-\frac1{d}}\big)n^{-\frac12+(\frac1p-\frac12)_+}.
\end{equation}Moreover, the upper estimate of \eqref{4.16} holds for all doubling weights.

Furthermore, if $w_{\boldsymbol\kappa,\mu}^{\Omega}$ is a Dunkl weight with the critical index $s_{w_{\boldsymbol\kappa,\mu}^{\Omega}}$ on $\Oz^d$, then for $r>s_{w_{\boldsymbol\kappa,\mu}^{\Omega}}/p$, we have
$$
e_n^{\rm ran}\left({\rm Int}_{\Oz^{d},w_{\boldsymbol\kappa,\mu}^{\Omega}}; BW_p^r(w_{\boldsymbol\kappa,\mu}^{\Omega})\right)\asymp n^{-\frac r{d}-\frac12+\left(\frac1p-\frac12\right)_+}.
$$
\end{thm}

The proof of upper estimates of Theorem \ref{thm4.2} is based on the relations between sphere, ball, and simplex, which is easy to be checked.
\begin{proof}[Proof of upper estimates.]\

 Similar to the proof of Theorem \ref{thm4.1}, it suffices to show the upper estimate of \eqref{4.16} for doubling weights. To this end, let $w$ be a doubling weight on $\Omega^d$  and   $f\in B_\gamma^\Theta(L_{p,w}(\Oz^d))$. Then, by \eqref{4.11} we have  $ F=f\circ\psi\in B_\gamma^\Theta(L_{p,T_\Oz w}(\mathbb S^d))$.
It follows from \eqref{3.23} that there exist a randomized algorithm $(A_n^\omega)_{\omega\in\mathcal F}$ for $C(\mathbb S^d)$ as
$$A_n^\oz(F)=\varphi_n^\oz\left[F({\bf x}_{1,\oz}^*),\dots,F({\bf x}_{n,\oz}^*)\right]$$ such that
$$\mathbb E_\oz\left|\int_{\mathbb S^d}F(\bar{{\bf x}})\,T_{\Oz}w(\bar{{\bf x}}){\rm d}\sigma(\bar{{\bf x}})
-A_n^{\oz}( F)\right|\lesssim
\Theta\big(n^{-\frac1{d}}\big)n^{-\frac12+(\frac1p-\frac12)_+}.$$
Let ${\bf y}_{j,\omega}^*=\psi({\bf x}_{j,\omega}^*)\in\Omega^d,j=1,\dots,n$, and
$$\tilde A_n^\omega(f):=A_n^\omega(f\circ \psi)= \varphi_n^\oz\left[f({\bf y}_{1,\oz}^*),\dots,f({\bf y}_{n,\oz}^*)\right].$$Then
$(\tilde A_n^\omega)_{\omega\in\mathcal F}$ can be viewed as a randomized algorithm for $C(\Omega^d)$, and by \eqref{4.2} we have
$$\mathbb E_\oz\left|\int_{\Oz^d}f({\bf x})w({\bf x}){\rm d}{\bf x}-\tilde A_n^{\oz}(f)\right|\lesssim \Theta\big(n^{-\frac1{d}}\big)n^{-\frac12+(\frac1p-\frac12)_+},$$
which yields the desired estimate.

The proof is finished.
\end{proof}

The proof of lower estimates of Theorem \ref{thm4.2} is almost
identical to that of the lower estimate of the sphere. For our
purpose, we also need to construct the fooling functions. The
argument depends on  an analog of Lemma
\ref{lem44} on $\Omega^d$. We
only give the construction of the fooling functions and the detail proof of Theorem
\ref{thm4.2} is omitted.

In the following, we consider $1\le p\le\infty$, $0<\gamma\le\infty$, $r>0$, and $\Theta_1(t)$, $\Theta(t)=t^r\Theta_1(t)\in\Phi_s^*$.
Let $\boldsymbol\kappa:=(\kappa_1,\dots,\kappa_m)\in[0,+\infty)^m$, $\mu\ge0$, and ${\bf v}_j\in\mathbb S^{d}$, $j=1,\dots,m$. Recall the product weight $w_{\boldsymbol\kappa,\mu}^\Omega$ on $\Omega^d$ given by \eqref{4.2-0}.
Then
$$T_{\Oz}w_{\boldsymbol\kappa,\mu}^{\Oz}(\bar{{\bf x}})=\prod_{j=1}^{m+1}\left|\langle\bar{{\bf x}},\bar{{\bf v}}_j\rangle\right|^{2\kappa_j}=h_{\boldsymbol\kappa,\mu}^2(\bar{\bf x})\equiv w_{\boldsymbol\kappa,\mu}^{\mathbb S}(\bar{{\bf x}}).$$
Let $\varphi$ be a nonnegative
$C^\infty$-function on $\mathbb{R}$ supported in $[0,1]$
and being equal to 1 on $[0,1/2]$, and let $N$ be a sufficient large integer with $N^d\asymp n$.

\subsubsection*{Construct fooling functions.}

First we consider the case of $\Omega^d=\mathbb B^d$. Using the same method in Subsection 3.2, we can choose
$\{\bar {\bf x}_j\}_{j=1}^{4n}$ in the upper sphere $$\mathbb S_+^d:=\{{\bf x}=(x_1,\dots,x_{d+1})\in\mathbb S^d:x_{d+1}\ge0\}$$  such that
 $$\bar\varphi_j(\bar{\bf x}):=\varphi(Nd_{\mathbb S}\left(\bar{\bf x},\bar{\bf x}_j)\right),\ j=1,\dots,4n$$have the following properties.
\begin{itemize}
  \item [(i)] The supports ${\rm supp}(\bar\varphi_j)\subset{\rm c}(\bar{\bf x}_j,1/N),j=1,\dots,4n$ are disjoint, that is
 $${\rm supp}(\bar\varphi_i)\bigcap {\rm supp}(\bar\varphi_j)=\emptyset,\ {\rm for}\ i\neq
j.$$

  \item [(ii)] For any $1\le p<\infty$, the $L_p(w_{\boldsymbol\kappa,\mu}^{\mathbb S})$ norm of $\bar\varphi_j$ is
  $$\|\bar\varphi_j\|_{L_p(w_{\boldsymbol\kappa,\mu}^{\mathbb S})}\asymp\Big(\int_{{\rm c}(\bar{\bf x}_j,\frac1N)}|\varphi(Nd_{\mathbb S}\left(\bar{\bf x},\bar{\bf x}_j)\right)|^p{\rm d}\sigma(\bar{\bf x})\Big)^{1/p}\asymp  n^{-1/p},
$$
and the $L_{\infty}$ norm of $\bar\varphi_j$ is
  $$\|\bar\varphi_j\|_{L_{\infty}(\mathbb S^d)}=\sup_{\bar{\bf x}\in{\rm c}(\bar{\bf x}_j,\frac1N)}|\varphi(Nd_{\mathbb S}\left(\bar{\bf x},\bar{\bf x}_j)\right)|\asymp1.
$$

  \item [(iii)]For any $\bar f_{\boldsymbol\az}:=\sum_{j=1}^{4n}
\az_j\bar\varphi_j$ with ${\boldsymbol\az}=(\az_1,\cdots,\az_{4n})\in\mathbb{R}^{4n}$, we have
  $$\|\bar f_{\boldsymbol\az}\|_{L_p(w_{\boldsymbol\kappa,\mu}^{\mathbb S})}\asymp
 n^{-1/p}\,\|{\boldsymbol\az}\|_{\ell_p^{4n}},$$
and\begin{itemize}
  \item for any $r>0$,
$$\|\bar f_{{\boldsymbol\az}}\|_{W_p^r(w_{\boldsymbol\kappa,\mu}^{\mathbb S})}\lesssim n^{r/d-1/p}\,\|{\boldsymbol\az}\|_{\ell_p^{4n}},
$$whenever $w_{\boldsymbol\kappa,\mu}^{\mathbb S}$ is a Dunkl weight on $\mathbb S^d$;
 \item for any $r>0$,
$$\|\bar f_{{\boldsymbol\az}}\|_{B_\gamma^\Theta(L_p(w_{\boldsymbol\kappa,\mu}^{\mathbb S}))}\lesssim \Theta(n^{-1/d})^{-1}n^{-1/p}\,\|{\boldsymbol\az}\|_{\ell_p^{4n}},
$$whenever $w_{\boldsymbol\kappa,\mu}^{\mathbb S}$ is a product weight on $\mathbb S^d$.
\end{itemize}
\end{itemize}

Clearly, the mapping
\begin{align*}
  \psi_1:\,\mathbb S_+^d\to \mathbb B^d,\quad({\bf x},x_{d+1})\mapsto{\bf x}
\end{align*} is bijection, which leads that its inverse $\psi_1^{-1}:\,\mathbb B^d\to \mathbb S_+^d$ exists. Set
$$\varphi_j({\bf x}):=\varphi(Nd_{\mathbb B}\left({\bf x},{\bf x}_j)\right),\ j=1,\dots,4n$$
and $$F_0:=\Big\{f_{\boldsymbol\az}:=\sum_{j=1}^{4n}
\az_j\varphi_j:\,
{\boldsymbol\az}=(\az_1,\cdots,\az_{4n})\in\mathbb{R}^{4n}\Big\},$$where ${\bf x}_j:=\psi_1(\bar{\bf x}_j),j=1,\dots,4n$. Then
$$\varphi_j({\bf x})=\varphi(Nd_{\mathbb B}\left(\psi_1(\bar{\bf x}),\psi_1(\bar{\bf x}_j)\right)=\varphi(Nd_{\mathbb S}\left(\bar{\bf x},\bar{\bf x}_j)\right)=\bar\varphi_j(\bar{\bf x}),\ j=1,\dots,4n.$$
Therefore, from \eqref{4.2}, \eqref{4.3} and the above properties of $\{\bar\varphi_j\}_{j=1}^{4n}$, we derive that
 for $f_{\boldsymbol\az}\in F_0$ with $\boldsymbol\az=(\az_1,\dots,\az_{4n})\in\mathbb R^{4n}$,
\begin{itemize}
  \item[(i)] if $w_{\boldsymbol\kappa,\mu}^{\mathbb B}$ is a Dunkl weight on ${\mathbb B}^d$, then
$$\| f_{{\boldsymbol\az}}\|_{W_p^r(w_{\boldsymbol\kappa,\mu}^{{\mathbb B}})}\lesssim n^{r/d-1/p}\,\|{\boldsymbol\az}\|_{\ell_p^{4n}};
$$
  \item[(ii)] if $w_{\boldsymbol\kappa,\mu}^{\mathbb B}$ is a product weight on ${\mathbb B}^d$,  then
$$\| f_{{\boldsymbol\az}}\|_{B_\gamma^\Theta(L_p(w_{\boldsymbol\kappa,\mu}^{\mathbb B}))}\lesssim \Theta(n^{-1/d})^{-1}n^{-1/p}\,\|{\boldsymbol\az}\|_{\ell_p^{4n}}.
$$\end{itemize}

Next we turn to consider the case of $\Omega^d=\mathbb T^d$. It can be transferred to the case of $\Omega^d=\mathbb B^d$ by the bijection
\begin{align*}
  \psi_2:\,\mathbb B_+^d\to \mathbb T^d,\quad{\bf x}\mapsto{\bf x}^2,
\end{align*}where $\mathbb B_+^d:=\{{\bf x}=(x_1,\dots,x_d)\in\mathbb B^d:\,x_1,\dots,x_d\ge0\}$. Similar to the case of $\Omega^d=\mathbb S^d$, we can choose
$\{\tilde {\bf x}_j\}_{j=1}^{4n}$ in $\mathbb B_+^d$  such that
 $$\tilde\varphi_j(\tilde{\bf x}):=\varphi(Nd_{\mathbb B}\left(\tilde{\bf x},\tilde{\bf x}_j)\right),\ j=1,\dots,4n$$have the following properties.
\begin{itemize}
  \item [(i)] The supports ${\rm supp}(\tilde\varphi_j)\subset C_{\mathbb B}(\tilde{\bf x}_j,1/N),j=1,\dots,4n$ are disjoint, that is
 $${\rm supp}(\tilde\varphi_i)\bigcap {\rm supp}(\tilde\varphi_j)=\emptyset,\ {\rm for}\ i\neq
j.$$

  \item [(ii)] For any $1\le p<\infty$, the $L_p(w_{\boldsymbol\kappa,\mu}^{\mathbb B})$ norm of $\tilde\varphi_j$ is
  $$\|\tilde\varphi_j\|_{L_p(w_{\boldsymbol\kappa,\mu}^{\mathbb B})}\asymp\Big(\int_{C_{\mathbb B}(\tilde{\bf x}_j,\frac1N)}\frac{|\varphi(Nd_{\mathbb B}\left(\tilde{\bf x},\tilde{\bf x}_j)\right)|^p}{\sqrt{1-\|\tilde{\bf x}\|^2}}{\rm d}\tilde{\bf x}\Big)^{1/p}\asymp  n^{-1/p},
$$
and the $L_{\infty}(\mathbb B^d)$ norm of $\tilde\varphi_j$ is
  $$\|\tilde\varphi_j\|_{L_{\infty}(\mathbb B^d)}=\sup_{\tilde{\bf x}\in C_{\mathbb B}(\tilde{\bf x}_j,\frac1N)}|\varphi(Nd_{\mathbb B}\left(\tilde{\bf x},\tilde{\bf x}_j)\right)|\asymp1.
$$

  \item [(iii)] For any $\tilde f_{\boldsymbol\az}:=\sum_{j=1}^{4n}
\az_j\tilde\varphi_j$ with ${\boldsymbol\az}=(\az_1,\cdots,\az_{4n})\in\mathbb{R}^{4n}$, we have
  $$\|\tilde f_{\boldsymbol\az}\|_{L_p(w_{\boldsymbol\kappa,\mu}^{\mathbb B})}\asymp
 n^{-1/p}\,\|{\boldsymbol\az}\|_{\ell_p^{4n}},$$
and\begin{itemize}
  \item for any $r>0$,
$$\|\tilde f_{{\boldsymbol\az}}\|_{W_p^r(w_{\boldsymbol\kappa,\mu}^{\mathbb B})}\lesssim n^{r/d-1/p}\,\|{\boldsymbol\az}\|_{\ell_p^{4n}},
$$whenever $w_{\boldsymbol\kappa,\mu}^{\mathbb B}$ is a Dunkl weight on $\mathbb B^d$;
 \item for any $r>0$,
$$\|\tilde f_{{\boldsymbol\az}}\|_{B_\gamma^\Theta(L_p(w_{\boldsymbol\kappa,\mu}^{\mathbb B}))}\lesssim \Theta(n^{-1/d})^{-1}n^{-1/p}\,\|{\boldsymbol\az}\|_{\ell_p^{4n}},
$$whenever $w_{\boldsymbol\kappa,\mu}^{\mathbb B}$ is a product weight on $\mathbb B^d$.
\end{itemize}
\end{itemize}
 Set
$$\varphi_j({\bf x}):=\varphi(Nd_{\mathbb T}\left({\bf x},{\bf x}_j)\right),\ j=1,\dots,4n$$
and $$F_0:=\Big\{f_{\boldsymbol\az}:=\sum_{j=1}^{4n}
\az_j\varphi_j:\,
{\boldsymbol\az}=(\az_1,\cdots,\az_{4n})\in\mathbb{R}^{4n}\Big\},$$where ${\bf x}_j:=\psi_2(\tilde{\bf x}_j),j=1,\dots,4n$. Then
$$\varphi_j({\bf x})=\varphi(Nd_{\mathbb T}\left(\psi_2(\tilde{\bf x}),\psi_2(\tilde{\bf x}_j)\right)=\varphi(Nd_{\mathbb B}\left(\tilde{\bf x},\tilde{\bf x}_j)\right)=\tilde\varphi_j(\tilde{\bf x}),\ j=1,\dots,4n.$$
Therefore, from \eqref{4.2}, \eqref{4.6} and the above properties of $\{\tilde\varphi_j\}_{j=1}^{4n}$, we derive that
 for $f_{\boldsymbol\az}\in F_0$ with $\boldsymbol\az=(\az_1,\dots,\az_{4n})\in\mathbb R^{4n}$,
\begin{itemize}
  \item[(i)] if $w_{\boldsymbol\kappa,\mu}^{\mathbb T}$ is a Dunkl weight on ${\mathbb T}^d$, then
$$\| f_{{\boldsymbol\az}}\|_{W_p^r(w_{\boldsymbol\kappa,\mu}^{{\mathbb T}})}\lesssim n^{r/d-1/p}\,\|{\boldsymbol\az}\|_{\ell_p^{4n}};
$$
  \item[(ii)] if $w_{\boldsymbol\kappa,\mu}^{\mathbb T}$ is a product weight on ${\mathbb T}^d$,  then
$$\| f_{{\boldsymbol\az}}\|_{B_\gamma^\Theta(L_p(w_{\boldsymbol\kappa,\mu}^{\mathbb T}))}\lesssim \Theta(n^{-1/d})^{-1}n^{-1/p}\,\|{\boldsymbol\az}\|_{\ell_p^{4n}}.
$$\end{itemize}

\section{Concluding Remarks}\label{sect5}

In this paper we explore the numerical integration $${\rm
Int}_{\Omega^d,w}(f)=\int_{\Omega^d}f({\bf x})w({\bf x}){\rm
d}\mathsf m({\bf x}) $$ for a weighted function class using the
deterministic and randomized algorithoms, where $\Omega^d$ denotes
the  sphere $\mathbb S^d$, the  ball $\mathbb B^d$, or the
standard simplex $\mathbb T^d$. The corresponding orders of
optimal quadrature errors for weighted Sobolev classes with Dunkl
weights were obtained. However, for the weighted Besov class of
generalized smoothness with a doubling weight $w$, we derived that
for $r>s_w/p$ and $1\le p\le\infty$,
\[
e_n^{\rm det}\left({\rm Int}_{\Oz^d,w}; BB_{\gamma}^\Theta(L_{p,w}(\Oz^d))\right)\lesssim \Theta\big(n^{-\frac1{d}}\big),
\]which is optimal whenever $w$ is an $A_\infty$ weight on $\Omega^d$, and
\[
e_n^{\rm ran}\left({\rm Int}_{\Oz^d,w}; BB_{\gamma}^\Theta(L_{p,w}(\Oz^d))\right)\lesssim \Theta\big(n^{-\frac1{d}}\big)n^{-\frac12+(\frac1p-\frac12)_+}.
\]which is optimal whenever $w$ is the product weight on $\Omega^d$.
We believe that the aforementioned orders are optimal  for all doubling weights, though verifying this may require  different techniques.
\vskip3mm
\noindent{\bf Acknowledgment}  Jiansong Li and Heping Wang  were
supported by the National Natural Science Foundation of China
(Project no. 12371098).

\end{document}